\documentclass[12pt, reqno]{amsart}
\pdfoutput=1

\usepackage{amssymb, amsthm, amsmath, amsfonts, bbm}
\usepackage{enumitem}
\usepackage{array}	
\usepackage{lineno}

\usepackage{hyperref}
\usepackage{esvect}

\usepackage{anysize}
\usepackage[pdftex]{graphicx}

\usepackage[makeroom]{cancel} 

\newcommand*\patchAmsMathEnvironmentForLineno[1]{%
  \expandafter\let\csname old#1\expandafter\endcsname\csname #1\endcsname
  \expandafter\let\csname oldend#1\expandafter\endcsname\csname end#1\endcsname
  \renewenvironment{#1}%
     {\linenomath\csname old#1\endcsname}%
     {\csname oldend#1\endcsname\endlinenomath}}%
\newcommand*\patchBothAmsMathEnvironmentsForLineno[1]{%
  \patchAmsMathEnvironmentForLineno{#1}%
  \patchAmsMathEnvironmentForLineno{#1*}}%
\AtBeginDocument{%
\patchBothAmsMathEnvironmentsForLineno{equation}%
\patchBothAmsMathEnvironmentsForLineno{align}%
\patchBothAmsMathEnvironmentsForLineno{flalign}%
\patchBothAmsMathEnvironmentsForLineno{alignat}%
\patchBothAmsMathEnvironmentsForLineno{gather}%
\patchBothAmsMathEnvironmentsForLineno{multline}%
}

  \evensidemargin
\oddsidemargin
\newcommand{\supp}{\mathop{\mathrm{supp}}\nolimits}

\newcommand{\sgn}{\mathop{\mathrm{sgn}}\nolimits}

\makeatletter
\DeclareFontFamily{U}{tipa}{}
\DeclareFontShape{U}{tipa}{m}{n}{<->tipa10}{}
\newcommand{\arc@char}{{\usefont{U}{tipa}{m}{n}\symbol{62}}}%

\newcommand{\arc}[1]{\mathpalette\arc@arc{#1}}

\newcommand{\arc@arc}[2]{%
  \sbox0{$\m@th#1#2$}%
  \vbox{
    \hbox{\resizebox{\wd0}{\height}{\arc@char}}
    \nointerlineskip
    \box0
  }%
}
\makeatother

\DeclareMathOperator{\conv}{conv} 
 
\DeclareMathOperator{\im}{im}

\DeclareMathOperator{\cl}{cl}

\DeclareMathOperator{\rint}{rint}
\DeclareMathOperator{\intr}{int}

\newtheorem*{def*}{Definition}

\newtheorem*{rem*}{Remark}
\newtheorem*{cor*}{Corollary}

\newtheorem{prop}{Proposition}

\newtheorem{lem}{Lemma}
\newtheorem{sublem}{Sublemma}
\newtheorem*{lem1'}{Lemma $\mathbf{1^\prime}$}

\newtheorem{theorem}{Theorem}

\theoremstyle{definition}

\theoremstyle{remark}
\newtheorem{remark}{Remark}

\def\R{\mathbb{R}}

\def\N{\mathbb{N}}
\def\E{\mathbb{E}}
\def\P{\mathbb{P}}

\def\S{\mathbb{S}}
\def\I{\mathbbm{1}}

\newcounter{fig}
\newcommand{\f}{\refstepcounter{fig} Fig. \arabic{fig}. }

\begin{document}

\ifpdf
\DeclareGraphicsExtensions{.pdf, .jpg, .tif, .mps}
\else
\DeclareGraphicsExtensions{.eps, .jpg, .mps}
\fi


\title[The isoperimetric problem for convex hulls and rate functionals]{The isoperimetric problem for convex hulls and the large deviations rate functionals of random walks}

\author{Vladislav~Vysotsky}
\address{Vladislav~Vysotsky, University of Sussex}
\email{v.vysotskiy@sussex.ac.uk}

\begin{abstract}
We study the asymptotic behaviour of the most likely trajectories of a planar random walk that result in large deviations of the area of their convex hull. If the Laplace transform of the increments is finite on $\R^2$, such a scaled limit trajectory $h$ solves the inhomogeneous anisotropic isoperimetric problem for the convex hull, where the usual length of $h$ is replaced by the large deviations rate functional $\int_0^1 I(h'(t)) dt$ and $I$ is the rate function of the increments. Assuming that the distribution of increments is not supported on a half-plane, we show that the optimal trajectories are  convex and satisfy the Euler--Lagrange equation, which we solve explicitly for every $I$. The shape of these trajectories resembles the optimizers in the isoperimetric inequality for the Minkowski plane, found by Busemann~(1947).
\end{abstract}


\maketitle

\section{Introduction} \label{Sec: intro}
Convex hulls of random walks in $\R^d$ have been studied since the 1960s and especially intensively over the past decade. The main objects of interest  typically are the quantitative characteristics of the convex hull such as the surface area, volume, diameter, number of faces, etc., or the characteristics of the shape such as the probability that the convex hull does not contain the origin. We refer to~\cite{AkopyanVysotskyProbab} and~\cite{KVZChambers} for accounts of results and references. 

Let $(S_n)_{n \ge 1}$, where $S_n = X_1 + \ldots + X_n$, be a  random walk on the plane with independent identically distributed increments $X_1, X_2, \ldots$. Denote by $A_n$ the {\it area of the convex hull} of $0, S_1, \dots, S_n $. In this paper we aim to identify the most likely paths of the walk resulting in atypically large values of this quantity. We always assume that $X_1$ has finite Laplace transform and is not supported on a line passing through the origin, that is $\E e^{u \cdot X_1} <\infty$ and $\P(u \cdot X_1 =0)<1$ for every $u \in \R^2$. The latter assumption excludes the trivial case $A_n=0$. 

The typical behaviour of the area was described by~Wade and Xu~\cite{WadeXu2}. As a consequence of the invariance principle, they showed that  $A_n/n^a$ converge weakly to a non-degenerate limit as $n \to \infty$, where $a=1$ if $\mu := \E X_1$ is zero and $a=3/2$ otherwise. As for atypically large values of $A_n$, Akopyan and Vysotsky~\cite[Theorem~2.11]{AkopyanVysotskyProbab} showed that 
the random variables $A_n/n^2$ satisfy the large deviations principle with speed $n$ and the rate function $\mathcal J_A$ defined below in~\eqref{eq: rate func I_A}.  In particular, this implies that for any continuity point $a \ge 0$ of $\mathcal J_A$, we have
\begin{equation} \label{eq: LDP A}
\lim_{n \to \infty} \frac 1n \log \P(A_n \ge an^2) = -\mathcal J_A(a).
\end{equation}
The scaling $1/n^2$  for  the area of the convex hull arises from the large deviations scaling $1/n$ of the trajectory $S_1, \ldots, S_n$, as can be seen in~\eqref{eq: LDP conditional} below.

To define $\mathcal J_A$, we need the following notation. The {\it cumulant generating function} of $X_1$ is $K(u):=\log \E e^{u \cdot X_1}$ for $u \in \R^2$. It is known to be convex. The {\it  rate function} $I$ of $X_1$ is the convex conjugate of $K$, that is 
\[
I(v):= \sup_{u \in \R^2} ( u \cdot v - \log \E e^{u \cdot X_1}), \qquad v \in \R^2.
\] 
This is a convex function with values in $[0, +\infty]$ that satisfies $I(\mu)=0$. Furthermore, denote by $AC_0[0,1]$ the set of absolutely continuous planar curves (i.e.\ mappings $h:[0,1] \to \R^2$  with absolutely continuous coordinates $h_1$ and $h_2$) that satisfy $h(0)=0$. Denote by $A(h)$ the area of the convex hull of the image of a curve $h$. Finally, for $a \ge 0$ define
\begin{equation} \label{eq: rate func I_A}
\mathcal J_A(a):=\min_{\substack{h \in AC_0[0,1]:  \\ A(h) = a} } I_C(h), \quad \text{where} \quad I_C(h):= \int_0^1 I(h'(t)) dt.
\end{equation}
The minimum is always attained at some $h$ due to the {\it tightness} property of $I_C$ on $AC_0[0,1]$; see~\cite[Eq.~(4.12)]{AkopyanVysotskyProbab}. The function $I_C$ is the rate function for large deviations of trajectories of the random walk in the space $AC_0[0,1]$. For brevity, we will refer to $I_C$ as the {\it rate functional} or  {\it energy functional}. The function $\mathcal J_A$ satisfies $\mathcal J_A(0)=0$ and is strictly increasing on its effective domain $\{a \ge 0: \mathcal J_A(a) <\infty\}$.  

The goal of this paper is to find the minimizers of $\mathcal J_A$ in \eqref{eq: rate func I_A}. These minimizers define the asymptotic form of the most likely trajectories of the random walk resulting in the large deviations events $\{A_n \ge a n^2 \}$, as $n \to \infty$. More precisely, by \cite[Theorem~2.11]{AkopyanVysotskyProbab}, for every $\varepsilon > 0$ and every continuity point $a >0$ of $\mathcal J_A$ such that $\mathcal J_A(a)<\infty$, it is true that
\begin{equation} \label{eq: LDP conditional}
\lim_{n \to \infty} \P \Bigl(\max_{0 \le k \le n} \bigl | S_k/n - h(k/n) \bigr|  \le \varepsilon \text{ for some minimizer $h$ in \eqref{eq: rate func I_A}} \Bigl | \Bigr. A_n  \ge an^2 \Bigr) = 1;
\end{equation}
moreover, the rate of this convergence is exponential. 

The paper~\cite{AkopyanVysotskyProbab} explored the idea of finding the minimizers (and thus $\mathcal J_A$) using the  isoperimetric inequalities of Moran~\cite{Moran} and Pach~\cite{pach1978isoperimetric} for the area of the convex hull of a planar curve. This turned out to be possible when the distribution of $X_1$ was a linear image of either a rotationally invariant (with $\mu=0$) or a shifted standard Gaussian (with $\mu \neq 0$) distribution. 

In this paper we use a different approach, which covers a much wider class of distributions than in~\cite{AkopyanVysotskyProbab}. Our main result, Theorem~\ref{thm: variational}, shows that in the case where $X_1$ is not supported on a half-plane, every minimizer of $\mathcal J_A$ 
is obtained by certain shift, dilation, and rotation through $\pm \pi/2$ of the bijective parametrization $g$ of an arc of a level set of $K$ with centrally symmetric endpoints $g(0)=-g(1)$ and the speed $|g'|$ proportional to $|\nabla K(g)|$. To show this, we first employ a geometric argument, based on the procedure of {\it convexification} of planar curves, to show that the minimum in~\eqref{eq: rate func I_A} is  attained on convex curves; see Lemma~\ref{lem: convexification}. Hence by Green's formula, the constraint in \eqref{eq: rate func I_A} can we written as $\int_0^1 (h_1 h_2' - h_1'h_2)dt = 2a$. The question of finding $\mathcal J_A$ then becomes an isoperimetric-type problem of the classical variational calculus. We show  that the minimizers $h(t)$ are smooth and satisfy the Euler--Lagrange equations; see Lemma~\ref{lem: Lagrangian}. We were able to solve these equations explicitly, which was quite unexpected.

It is easy to overlook that a general integral functional $F=\int_0^1 f(t, x(t), x'(t)) dt$ may have absolutely continuous minimizers  that  {\it do not} satisfy the Euler--Lagrange equation; see e.g.~\cite[Proposition~6.13]{Buttazzo+}. This equation is satisfied only under some a priori assumptions on a minimizer $x_*$, e.g.\ that it is $C^1$-smooth or that $x_*'$ is essentially bounded; see~\cite{ClarkeVinter} for an illuminating discussion. For example, the latter assumption is well-justified for a standard variational proof of Dido's isoperimetric inequality by solving the Euler--Lagrange equations, because every minimizer can be time-changed to have a constant speed  since this does not change the length and the enclosed area. However, this argument does not apply when $F$ is not homogeneous. We also note that the infimum of $F$ over absolutely continuous functions can be strictly smaller than the infimum over smooth or even over Lipshitz functions; this is referred to as the {\it Lavrentiev phenomenon}, see~\cite[Section~4.3]{Buttazzo+}.

Thus, in general one cannot hope of finding the minimizers of $F$ by solving the Euler--Lagrange equations. Additional assumptions are needed, and it took us a serious effort to find due references. We therefore think that the corresponding part of the paper may be of significant interest to the probabilists working on variational problems arising in the studies of large deviations. Further examples of such problems coming from applications can be found e.g.\ in the book~\cite{DeviationsBook}, where Appendix C also highlights the difficulties with variational calculus described above.  


To place our result in a wider context, we first note that the variational problem~\eqref{eq: rate func I_A} resembles the classical Dido's isoperimetric problem if we formally let $I(v)=|v|$, thus replacing $I_C(h)$ by the length of $h$. A more general isoperimetric problem for closed curves of fixed {\it Minkowski length} was solved by Busemann~\cite{Busemann}; more on this below. His result can be also stated in terms of the surface energy, which is minimized on the so-called {\it Wulff sets}, see~\cite{Fonseca}. For further isoperimetric inequalities in non-Euclidean metrics, without aiming to be comprehensive we refer to the surveys~\cite{MartiniMustafaev, Osserman} and to the quantitative-type results in~\cite{Figalli+, Osserman79}. All these isoperimetric inequalities concern the area bounded by closed curves. The versions of Dido's problem for convex hulls was solved by Moran~\cite{Moran} for the curves with free endpoints and by Pach~\cite{pach1978isoperimetric} for the fixed endpoints; further results are available in~\cite{Tilli} and the survey~\cite{Zalgaller}. 

The main difference of our isoperimetric problem from the ones above is that the rate function $I$ is not homogeneous, hence the time-changed minimizers are no longer optimal. A simplest related inhomogeneous example is the Chaplygin optimal control problem (see~\cite[Section~1.6.4]{ATF}) 
of finding a closed planar path of an airplane moving in a constant wind field that encircles the maximum possible area at a unit time, possibly under additional constraints on the  velocity of the plane. Formally this corresponds to $I(v)=|v-\mu|$  for $v \in B$ and $I(v)=\infty$ for $v \not \in B$, where $B \subset \R^2$ is a closed set and $\mu \neq 0$. 

We already commented above on the use of isoperimetric inequalities the context of convex hulls of random walks. The other applications of such inequalities can be found in the papers~\cite{Khoshnevisan, KuelbsLedoux}, which studied the laws of iterated logarithms for the convex hulls. This topic is closely related to the large deviations problem considered here. A recent work in this direction is~\cite{Cygan+}, which established the law of iterated logarithm for $A_n$.


\section{The main results}
For any closed set $C \subset \R^2$, denote by $A(C)$ the area of the convex hull $\conv( C)$ of $C$; then $A(h)=A(h([0,1]))$ for $h \in AC_0[0,1]$ in \eqref{eq: rate func I_A}. For any $u \in \R^2$, let $u^\bot$ stand for $u$ rotated counterclockwise through $\pi/2$ about the origin. Denote by $\supp(X)$ the topological support of the distribution of a random variable $X$. 

The case of our main interest is where $\conv(\supp(X_1))$ is two-dimensional. However, for the purpose of completeness we will also consider the case where $X_1$ is supported on a straight line not passing through the origin. In this case the trajectory of $S_n$ can be regarded as the graph of a one-dimensional random walk. 

To state our results, we shall introduce further notation for the two-dimensional case, where we will additionally assume that $0 \in \intr(\conv(\supp(X_1)))$; equivalently, there is no non-zero $u \in \R^2$ such that $u \cdot X_1 \ge 0$ a.s. This implies that the cumulant generating function $K$ of $X_1$ increases to infinity in every direction, hence $K$ is bounded from below and has bounded level sets because it is convex; see~\cite[Theorem~8.7]{Rockafellar}. 



For every $\alpha > 0$, $\ell \in \S$, and $\tau \in \{+, - \}$, denote by
\begin{equation} \label{eq: E^pm def}
E(\alpha):=A ( K^{-1}(\alpha)), \qquad E^\tau(\alpha, \ell): = A \bigl( \{ u \in K^{-1}(\alpha): \tau u \cdot \ell^\bot \ge 0 \}\bigr)
\end{equation}
the area of the sub-level set $K^{-1}((-\infty, \alpha])$ and its two subsets that lie to the different sides of the line $\ell \R$. This line intersects $K^{-1}(\alpha)$ at two points since $K(0)=0$. We also put
\begin{equation} \label{eq: lambda^pm}
\lambda_{\alpha, \ell}^\tau:=\int_{\{ u \in K^{-1}(\alpha): \tau u \cdot \ell^\bot \ge 0 \}} \frac{1}{|\nabla K(s)|} \sigma(ds),
\end{equation}
where $\sigma$ stands for the one-dimensional Hausdorff measure normalized such that $\sigma(K^{-1}(\alpha))$ is the length of $K^{-1}(\alpha)$.
This quantity has the following meaning: by the coarea formula, which is a ``curvilinear'' generalization of Fubini's theorem (see \cite[Theorem~3.13(ii)]{EvansGariepy}), 
\begin{equation} \label{eq: coarea}
\frac{\partial}{\partial \alpha } E^\tau(\alpha, \ell) = \lambda_{\alpha, \ell}^\tau.
\end{equation}
If the distribution of $X_1$ is  centrally symmetric, then $E^\tau(\alpha, \ell) = \frac12 E(\alpha) $ for every $\ell \in \S$.

Furthermore, denote by $g_{\alpha, \ell}^\tau$ the bijective continuous parametrization  of the arc $\{ u \in K^{-1}(\alpha) : \tau u \cdot \ell^\bot \ge 0 \}$ by $[0,1]$ that  has orientation $\tau$ and satisfies
\begin{equation} \label{eq: parametrization 2} 
\int_{g_{\alpha, \ell}^\tau([0,t])} \frac{1}{|\nabla K(s)|} \sigma(ds) = t \lambda_{\alpha, \ell}^\tau, \quad t \in [0,1].
\end{equation}
We will comment below on existence of $g_{\alpha, \ell}^\tau$ and prove its smoothness. Note that the half-lines $\ell \R_+$ and $\ell \R_-$ intersect $K^{-1}(\alpha)$ at the respective points $g_{\alpha, \ell}^\tau(0)$ and $g_{\alpha, \ell}^\tau(1)$, and by the change of variable formula, \eqref{eq: parametrization 2} is equivalent to
\begin{equation} \label{eq: parametrization} 
| (g_{\alpha, \ell}^\tau)'(t)| = \lambda_{\alpha, \ell}^\tau |\nabla K(g_{\alpha, \ell}^\tau(t))|, \quad t \in [0,1].
\end{equation}

We now state the main result of the paper.

\begin{theorem} \label{thm: variational}
Suppose that $\E e^{u \cdot X_1}< \infty$ for every $u\in \R^2$.

1. Assume that $\conv( \supp(X_1)) = \R^2$. Then for any $a>0$, every minimizer $h$ in \eqref{eq: rate func I_A} is a $C^\infty$-smooth curve of the form
\begin{equation} \label{eq: optimal curve}
h(t)=-\frac{1}{\tau \lambda_{\alpha, \ell}^\tau}\big(g_{\alpha, \ell}^\tau(t) - g_{\alpha, \ell}^\tau(0)\big)^\bot, \quad t \in [0,1],
\end{equation}
where $\alpha>0$, $\tau \in \{+,-\}$, $\ell \in \S$ are such that 
\[
\frac{\partial}{\partial \alpha} \sqrt{E^\tau(\alpha, \ell)}= \frac{1}{2 \sqrt a}
\] 
and the two-point set $K^{-1}(\alpha) \cap \ell \R$ is centrally symmetric. Moreover, if the distribution of $X_1$ is centrally symmetric, i.e.\ $X_1 \stackrel{d}{=} -X_1$, then $\alpha$ is the unique positive solution to $(\sqrt{ E(\alpha)})'=1/\sqrt{2a}$ and every curve of the form~\eqref{eq: optimal curve} with arbitrary $\tau$ and $\ell$ is a minimizer.

2. Assume that $\conv( \supp(X_1)) = \{\mu_1\} \times \R$ for a real $\mu_1 \neq 0$. Then for any $a >0$, the minimizers in \eqref{eq: rate func I_A} are  the two curves 
\begin{equation} \label{eq: h_pm}
h_\pm(t)=\Big(\mu_1 t, \frac{1}{\pm2 u_a} (K(0, \pm u_a(2t-1)) - K(0, \mp u_a) )\Big), \quad t \in [0,1], 
\end{equation}
where $u_a$ is the unique positive solution to $4a=|\mu_1| E'(u)$ with $E(u):=\int_{-1}^1 K(0,us)ds$.
\end{theorem}
Let us give general comments, and then give examples. 
\begin{enumerate}[leftmargin=*, label=\alph*)]
\item As usual with necessary conditions, in Part 1 the minimizers are of the form \eqref{eq: optimal curve}, but some curves of this form may not be minimizing. Therefore, the problem of finding $\mathcal J_A$  reduces to minimizing the energy functional $\int_0^1 I(h'(t)) dt$ over the curves of the form~\eqref{eq: optimal curve}. 

\item Let us explain how our theorem matches the isoperimetric inequality of Busemann~\cite{Busemann}. Consider the isoperimetric problem~\eqref{eq: rate func I_A} with $I$ substituted by 
\[
{\| \cdot \|}_B(v):= \inf\{r>0: v \in r B\},
\] 
which is the {\it Minkowski functional} of a compact convex set $B \subset \R^2$ such that $0 \in \intr B$. If $B$ is centrally symmetric, then $\| \cdot \|_B$ is a norm and $B$ is a centred unit ball in this norm. Busemann~\cite{Busemann} showed that the closed simple curves with a fixed Minkowski length $\int_0^1 {\| h'(t)\|}_B dt$ that enclose a maximal Euclidean area parametrize a translated dilation of $\partial B^\circ$ rotated through $\pm \pi/2$, where $B^\circ:= \cap_{u \in B} \{v \in \R^2: u \cdot v \le 1 \}$ is the polar set of $B$. To see the analogy with our result, recall that the convex conjugate ${\| \cdot \|}_B^*$ of ${\| \cdot \|}_B$ is $0$ on $B^\circ$ and $+\infty$ elsewhere. Then $\partial B^\circ$ is the boundary of the sub-level set ${\| \cdot \|}_B^*((-\infty, \alpha])$ for any $\alpha \ge 0$, and we see that $\partial B^\circ$ corresponds to the boundary $K^{-1}(\alpha)$ of the sub-level $K((-\infty, \alpha])$ since $K=I^*$.

\item The classical Cram\'er's large deviations principle for random walks in $\R^2$ is proved under the  assumption that  $\E e^{u \cdot X_1}< \infty$ for every $u$ in a neighbourhood of zero, instead of every $u \in \R^2$  as in our results. Proposition~4.1 in~\cite{AkopyanVysotskyProbab} established the large deviations principle for $A_n/n^2$ under this weaker assumption, which imposes additional essential difficulties. In this case,  the problem of finding the optimal trajectories of the random walk reduces to solving an isoperimetric problem in the class of (possibly discontinuous) curves of bounded variation equipped with a rate functional that is more complicated than $I_C$. This problem is solved in \cite{AkopyanVysotskyProbab} for walks with rotationally invariant distributions of increments. Naturally, in this case the Euler--Lagrange approach of the current paper can no longer be applied. 

\item The assumption $\conv( \supp(X_1)) = \R^2$ of Part 1, which means that there is no half-plane containing $\supp(X_1)$, is technical. It is  imposed to ensure that $I$ is finite and smooth on the whole of~$\R^2$. Essentially, this is needed only to apply the results of variational calculus used in the proof; see Theorem~\ref{thm: Cesari} in the Appendix. In some cases it is possible to remove this assumption using an approximation argument. We illustrate this by the following result, where we assume for simplicity that $X_1 \stackrel{d}{=} -X_1$. It still allows us to find $\mathcal J_A$ although we can no longer claim that all minimizers are as in~\eqref{eq: optimal curve}. For completeness, we also consider the one-dimensional case. As discussed in~\cite{Zalgaller}, in the subject of isoperimetric inequalities it is not uncommon to see additional assumptions, which presumably are excessive.  

\end{enumerate}

\begin{prop} \label{prop}
Suppose that $\E e^{u \cdot X_1}< \infty$ for every $u\in \R^2$. 

1. Assume that the distribution of $X_1$ is centrally symmetric and is not supported on a straight line. Then for any $a \in (0, a_{max})$, where $a_{max}:=\big( \lim_{\beta \to \infty} [\sqrt{2E(\beta)} ]' \big)^{-2}$, there is a minimizer $h$ in \eqref{eq: rate func I_A} that satisfies~\eqref{eq: optimal curve}.

2. Assume $X_1$ is supported on $\{\mu_1\} \times \R$ for a real $\mu_1 \neq 0$, and $X_1$ is not a constant a.s. Then for any $a \in (0, a_{max})$, where $a_{max}:=\frac14 \mu_1 \lim_{u \to \infty} E'(u)$, the functions $h_\pm$ in~\eqref{eq: h_pm} are among the minimizers in~\eqref{eq: rate func I_A}.
\end{prop}

Let us give examples to Theorem~\ref{thm: variational} and Proposition~\ref{prop}.

\begin{enumerate}[leftmargin=*, label=\alph*)]
\item Let the distribution of $X_1$ be rotationally invariant. Then $K^{-1}(\alpha)$ is a circle of radius $r_\alpha$ and $|\nabla K|= c_\alpha$ is constant on $K^{-1}(\alpha)$ for each $\alpha>0$. Then $\lambda_{\alpha, \ell}^\tau = \lambda_\alpha =\pi r_\alpha /c_\alpha $, and an optimal curve $h$ parametrizes a half-circle of radius $r_\alpha/\lambda_\alpha= c_\alpha/\pi$ with constant speed $c_\alpha$. From $2a= \pi(c_\alpha/\pi)^2$, we find that $c_\alpha=\sqrt{2 \pi a}$ and $\mathcal J_A(a)=I_C(h)=I_0(c_\alpha)$, where $I_0$ is the radial component of $I$ defined by $I_0(|u|)=I(u)$. This result was proved in~\cite{AkopyanVysotskyProbab}, where it was also extended to linear images of rotationally invariant distributions.

\item Let $X_1$ be Gaussian$(\mu, \Sigma)$ with a non-degenerate $\Sigma$ and $\mu \neq 0$. Let us assume that $\Sigma = Id$; reduction to the general case is explained in ~\cite[Remark~2.14]{AkopyanVysotskyProbab}. Then $K(u)=\frac12 |u|^2 + u \cdot \mu$. Since $K(u)=K(-u)$ only when $u \cdot \mu =0$, the four curves $g_{\alpha, \ell}^\tau$ in \eqref{eq: optimal curve} parametrize the two arcs of $K^{-1}(\alpha)$ lying to the different sides of the line $\mu^\bot \R$ (a subsequent computation will show that the larger arc is not optimal). These curves have a constant speed since $|\nabla K|$ is constant on the circle $K^{-1}(\alpha)$, so it is easy to write $g_{\alpha, \pm \mu^\bot}^\tau$ in the parametric form. It remains to compute the radius of $K^{-1}(\alpha)$ in terms of $a$, which was done in~\cite{AkopyanVysotskyProbab}. It follows that the minimizers $h$ are the constant speed parametrizations of the two $\mu$-axially symmetric circular arcs of radius $\frac{|\mu|}{2 \varphi_a \cos \varphi_a}$ and angle $2 \varphi_a$ starting at the origin and ending on the $\mu$-axis, where $\varphi_a \in (0, \pi/2)$ is the unique solution to 
\[
\frac{2\varphi - \sin 2 \varphi}{8 \varphi^2 \cos^2 \varphi} = \frac{a}{|\mu|^2}, 
\]
and $\mathcal J_A(a)=4 a \varphi_a - \frac12 |\mu|^2 \tan^2 \varphi_a$.

\item If $X_1$ is a degenerate Gaussian, assume that $X_1 \stackrel{d}{=}  (\mu_1, Y)$, where $\mu=(\mu_1, \mu_2)$, $Y$ is a Gaussian$(\mu_2, \sigma^2)$, and $\mu_1$ and $\sigma$ are non-zero constants. Then $K(0,u)= \frac12 \sigma^2 u^2 + \mu_2 u$ for real $u$, and it is easy to find from~\eqref{eq: h_pm} that the optimal trajectories are $h(t)= \mu t \pm 6a \mu_1^{-1}(0,  t^2-t)$ and $\mathcal J_A(a) = 6 a^2 \mu_1^{-2} \sigma^{-2}$. In the case $\mu_1=1$, this describes large deviations for the area of the convex hull of the graph of a one-dimensional random walk with Gaussian increments.  This result was presented in~\cite{AkopyanVysotskyProbab} without a full proof; the corresponding isoperimetric inequality was also obtained by Cygan et.\ al~\cite{Cygan+}.

\item Assume that $X_1 \stackrel{d}{=} (1, Y)$ with $\P(Y=\pm 1)=1/2$, so $(S_n)_{n \ge 1}$ is the graph of a symmetric simple  random walk. Then $K(0, u)=\log (\cosh u)$ and $E'(u)=\int_{-1}^1 s \tanh(us) ds$, hence $a_{max}=1/4$. The solution $u_a$ to $E'(u)=4a$ for $a \in (0, 1/4)$ is not explicit. We can only state that among the minimizers are the curves 
\[
h_{\pm a}(t)=\Big(t,  \frac{1}{ \pm 2 u_a} \log \Big(\frac{\cosh u_a(2t-1)}{\cosh u_a} \Big) \Big),
\] 
and they converge pointwise as $a \to 1/4-$ to the ``triangular'' curves defined by $h_{\pm1/4}(t):= (t, \mp t)$ for $t \in [0,1/2]$ and $h_{\pm1/4}(t):= (t,\pm (t-1))$ for $t \in [1/2,1]$. We have $\mathcal J_A(a)=I_C(h_a)$ for $a \in (0, 1/4)$, and 
$\mathcal J_A(1/4)=I_C(h_{1/4})=\log 2$ by the lower semi-continuity of $\mathcal J_A$ since $A(h_{\pm1/4})=1/4$. 
\end{enumerate}

\medskip

The rest of the paper is organized as follows. The proof of Theorem~\ref{thm: variational}  is based on two auxiliary Lemmas~\ref{lem: convexification} and~\ref{lem: Lagrangian}, presented in the following sections. The first lemma uses a geometric argument, which allows us to  convexify a minimizer in \eqref{eq: rate func I_A}, transforming it into a convex curve while preserving its optimality.  Lemma~\ref{lem: convexification} and its proof are of independent interest, due to their  geometric nature and possible applications to general energy functionals besides $I_C$; see Remark~\ref{rem}. The second lemma uses the results of variational calculus to show that a convex minimizer satisfies the Euler--Lagrange equations. This implies, combined with Lemma~\ref{lem: convexification}, that all minimizers are convex. In the last section we will prove Theorem~\ref{thm: variational}, essentially by solving the Euler--Lagrange equations, and prove Proposition~\ref{prop} by an approximation argument.  The necessary geometric and variational results are given in the Appendix. 

\section{The convexification lemma}

Let us introduce additional notation. Recall that a planar curve $h$ on $[0,1]$ is {\it convex} if its image belongs to the boundary of a convex set and $h(t_1)=h(t_2)$ for $t_1 \le t_2$ is possible only when $t_1=t_2$ or $t_1=0, t_2=1$.  The boundary of the convex hull of a convex  curve consists of the image of the curve and the line segment joining the end points. We say that a curve $h \in AC_0[0,1]$ is a {\it stopping convex curve} if a) its image, denoted by $\im h$, belongs to the boundary of a convex set $C$, and b) for any $0 \le t_1< t_2 \le 1$ such that $\{t_1, t_2\}\neq \{0,1\}$, the equality $h(t_1)=h(t_2)$ implies that $h(t)=h(t_1)$ for $t \in [t_1, t_2]$. If $A(h)>0$ and $h$ parametrizes the arc $\im h $ of $\partial C$ anti-clockwise (resp., clockwise), we say that $h$ has positive (resp., negative) orientation. Note that any stopping convex curve that is constant on no interval is convex. Denote by $AC_0^c[0,1]$ (resp., $AC_0^{sc}[0,1]$) the set of  convex (resp., stopping convex) curves in $AC_0[0,1]$.

For any $h =(h_1, h_2) \in AC_0[0,1]$, define the {\it signed area} enclosed by $h$ to be
\begin{equation} \label{eq:signed area}
\tilde{A}(h):= \frac12 \int_0^1 (h^\bot \cdot h') dt = \frac12 \int_0^1 (h_1 h_2' - h_1' h_2)dt .
\end{equation}
We will use Green's theorem (which in its standard form is stated for $C^1$-smooth closed simple curves) and an approximation argument to show that for stopping convex curves, 
\begin{equation} \label{eq: Green}
A(h) =  |\tilde{A}(h)|, \qquad h \in AC_0^{sc}[0,1].
\end{equation}
If $A(h)>0$, then the sign of $\tilde{A}(h)$ is the orientation of $h$. The meaning of the signed area for non-convex curves is clarified in Lemma~\ref{lem: signed} in the Appendix.
We will prove~\eqref{eq: Green} at Step 2 in the proof of Lemma~\ref{lem: convexification}. 


\begin{lem} 
\label{lem: convexification}
For any  $h \in AC_0[0,1] $ satisfying $A(h)>0$ and $I_C(h)<\infty$, there exists a convex curve $g \in AC_0^c [0,1]$ such that 
\begin{equation} \label{eq: g ineq}
A(h) \le A(g), \qquad |\tilde A(h)| \le A(g), \qquad I_C(g) \le I_C(h).
\end{equation}
Moreover, if $h$ has the property of being affine and non-constant on some non-empty subinterval of $[0,1]$, then $g$ can be chosen to satisfy the same property.
\end{lem}
\begin{remark} \label{rem}
The assertions of Lemma~\ref{lem: convexification} remain valid if $I_C$ is replaced by a functional $J_C(h):=\int_0^1 J(h'(t))dt$, defined for $h \in AC[0,1]$, where $J:\R^2 \to [0, \infty]$ is any lower semicontinuous convex function that attains its minimum. These are the only properties of $I$ used in the proof, and they are satisfied by $I$ under no assumptions on $X_1$.
\end{remark}

Our proof rests on the geometric procedure of {\it convexification} by B{\"o}r{\"o}czky et al.~\cite{boroczky1986maximal}, F{\'a}ry and Makai~\cite{fary1982isoperimetry}, and Pach~\cite{pach1978isoperimetric}. The novelty of Lemma~\ref{lem: convexification} is that it combines the inequalities for the areas with the energy estimate on $I_C$, while these authors considered the length functionals $J_C$ with homogeneous $J$. There are many other rearrangement-based transformations, for example Steiner's symmetrization, resulting in similar inequalities; see~\cite{Burchard}.  

\begin{proof}
The geometric procedure of convexification, described in the Appendix, transforms a polygonal line into a convex polygonal line. In terms of curves, it corresponds to transforming a piecewise linear planar curve into a (piecewise linear) stopping convex curve. This transformation does not change the energy of the curve and does not decrease the area and the signed area of its convex hull. We will combine these properties with the continuity properties of $A$, $\tilde A$, $I_C$ to show that the stopping convex curve $h^c$, obtained as the limit of convexifications of piecewise linear approximations of a curve $h \in AC_0[0,1]$, satisfies 
\begin{equation} \label{eq: h^c ineq}
A(h) \le A(h^c), \quad |\tilde A(h)| \le A(h^c), \quad I_C(h^c) \le I_C(h).
\end{equation} 
We will elaborate on the approximation argument since the geometric papers \cite{boroczky1986maximal, fary1982isoperimetry, pach1978isoperimetric} do not provide details. To finish the proof, we will transform $h^c$ (which we call a {\it convexification} of $h$) into a convex curve $g$ required\footnote{Existence of $h^c$ is already sufficient to prove Lemma~\ref{lem: Lagrangian}, which implies our main result, Theorem~\ref{thm: variational}. From this perspective Lemma~\ref{lem: convexification} is excessive. We gave it this stronger form for aesthetic reasons, to avoid using the artificial notion of stopping convex curves in the statement.}.

Step 1. Construction of a stopping convex curve $h^c$ that satisfies \eqref{eq: h^c ineq}.

Consider $\mathcal{D}_k:= \{[1-1/2^k,1] \} \cup \{[i/2^k, (i+1)/2^k )\}_{i=0}^{2^k-2}$, the sequence of increasing  partitions of $[0,1]$ into the dyadic intervals, for integer $k \ge 1$. Let $h_k$ be the piecewise linear functions on $[0,1]$ such that $h_k(0)=0$, $h_k(1)=h(1)$, and 
\[
h_k'(t) = 2^ k \sum_{i =0}^{2^k-1} \Bigl[h \Bigl( \frac{i+1}{2^k}\Bigr) -  h \Bigl( \frac{i}{2^k}\Bigr) \Bigr]  \I_{[i/2^k, (i+1)/2^k )}(t), \quad t \in (0,1).
\]
In other words, $h_k$ is the continuous function on $[0,1]$ defined by the linear interpolation between its values at the points $i/2^k$, where it coincides with $h$. Put $\mathcal{D}_\infty:= \cup_{k=1}^\infty \mathcal D_k$, and 
denote by $P$ the Lebesgue measure restricted to $\sigma(\mathcal{D}_\infty)$ and by $E$ the expectation with respect to $P$ regarded as a probability measure. Let us agree that in this proof, $L^1$ and `a.e.' refer to $P$. We have $h' \in L^1$, and it is easy to see that the sequence $(h_k')_{k \ge 1}$ is a martingale with respect to the filtration $\{\sigma(\mathcal{D}_k)\}_{k \ge 1}$ on the probability space $([0,1], \sigma(\mathcal{D}_\infty), P)$, and that $h_k'=E(h'| \sigma(\mathcal{D}_k))$ a.e. By L{\'e}vy's martingale convergence theorem (see Williams~\cite[Theorem~14.2]{Williams}), we have $h_k' \to h'$ a.e.\ and in $L^1$ as $k \to \infty$ and, moreover, the sequence $h_k'$ is uniformly integrable. Then $|h(t) - h_k(t)| \le \|h-h_k'\|_{L^1} \to 0$ for $t \in [0,1]$, and thus $h_k \to h$ uniformly. 

We claim that
\begin{equation} \label{eq: continuity}
A(h_k) \to A(h), \quad \tilde A(h_k) \to \tilde A(h), \quad I_C(h_k) \to I_C(h)
\end{equation}
as $k \to \infty$. The first assertion follows from Steiner's formula (see e.g.~\cite[Eq.~(A.3)]{AkopyanVysotskyProbab}) and the convergence $\conv(\im(h_k)) \to \conv(\im(h))$ in the Hausdorff distance, which itself follows from the uniform convergence $h_k \to h$. The second assertion in \eqref{eq: continuity} follows from
\begin{align} \label{eq: signed area -}
\Big | \tilde A(h)  - \tilde A(h_k) \Big | &= \frac12 \left |E (h^\bot \cdot h' - h_k^\bot \cdot h_k' )  \right | \notag \\
&\le  \left | E [ h^\bot \cdot ( h' - h_k')]  \right | + \left | E [(h^\bot - h_k^\bot ) \cdot h_k'] \right | \\
&\le {\| h \|}_\infty {\| h' - h_k'\|}_{L^1} + {\|h - h_k \|}_\infty {\| h_k'\|}_{L^1} \to 0. \notag
\end{align}
Furthermore, since $I$ is convex and $I_C(h)=\|I(h')\|_{L^1} <\infty$ by assumption of the lemma, by Jensen's inequality for conditional expectations, we have $I(h_k') \le E(I(h')|\sigma(\mathcal D_k))$ a.e.\ for every $k \ge 1$. By taking expectation, this gives $I_C(h_k)=E I(h_k') \le E[E(I(h')|\sigma(\mathcal D_k))] = I_C(h)$. On the other hand, 
by Fatou's lemma, lower semicontinuity of $I$, and the fact that $h_k' \to h'$ a.e.\ as $k \to \infty$,  we have 
$\liminf_k E I(h_k') \ge  E [\liminf_k I(h_k')] \ge E I(h')$, hence $\liminf_k I_C(h_k) \ge I_C(h)$, and thus $I_C(h_k) \to I_C(h)$ as claimed.   

Let $(h_k^c)_{ k \ge 1}$ be the piecewise linear functions that satisfy $h_k^c(0)=0$, $h_k^c(1)=h(1)$,  and 
\[
(h_k^c)'(t) = 2^ k \sum_{i =0}^{2^k-1} \Bigl[h \Bigl( \frac{\sigma_k(i)+1}{2^k}\Bigr) -  h \Bigl( \frac{\sigma_k(i)}{2^k}\Bigr) \Bigr]  \I_{[i/2^k, (i+1)/2^k )}(t), \quad t \in (0,1),
\]
where $\sigma_k$ is the unique permutation of the set $\{0, 1, \ldots, 2^k-1\}$ that arranges the vectors $2^k[h((\sigma_k(i)+1)/2^k) - h((\sigma_k(i))/2^k)]$, $i \in \{0, \ldots, 2^k-1\}$, in the following lexicographic order. First, order by increase of the angle (measured in the clockwise direction, with the convention that for the zero vector, the angle is $0$) between the vector and $h(1)$ if $h(1) \neq 0$ and any fixed direction if $h(1)=0$. Second, order by increase of norm, and third, order by increase of index $i$. 

By the construction, each $h_k^c$ is a stopping convex curve, and its image $\im(h_k^c)$ is a polygonal line, which is a convexification of the polygonal line $\im(h_k)$. Then 
\begin{equation} \label{eq: ineq convexif}
A(h_k^c) \ge A(h_k), \quad |\tilde A(h_k^c)| \ge |\tilde A(h_k)|, \quad I_C(h_k^c) = I_C(h_k),
\end{equation}
where the inequalities follow from Lemmas~\ref{lem:fary-convex},~\ref{lem: signed} and Proposition~\ref{prop:pach theorem} in the Appendix, and the equality holds true because $(h_k^c)'$ is constructed from $h_k'$ by reordering the intervals of constancy. 

By the same reasoning of reordering, the sequence $(h_k^c)'$ is uniformly integrable because so is the sequence $h_k'$. 
By a sequential weak compactness criterion in $L^1[0,1]$ (see Buttazzo et al.~\cite[Theorem~2.12]{Buttazzo+}), there exist a strictly increasing subsequence $(k_i)_{i \ge 1} \subset \N$ and an element of $AC_0[0,1]$, which we denote by $h^c$ and call a  convexification of $h$, such that $(h_{k_i}^c)' \to (h^c)'$ weakly in $L^1$ as $i \to \infty$. That is, $E [(h_{k_i}^c)' f] \to E[ (h^c)' f]$ for every $f \in L^\infty$. Hence $h_{k_i}^c \to h^c$ pointwise on $[0,1]$. Moreover, w.l.o.g.\ we can assume that this convergence is uniform:  $(h_k^c)_{k \ge 1}$ is equiabsolutely continuous by uniformly integrability of the sequence $(h_k^c)'$ (\cite[Theorem~2.12]{Buttazzo+}), hence $(h_k^c)_{k \ge 1}$ is equicontinuous and therefore equibounded by $h_k^c(0)=0$, and we can apply the Arzel{\'a}--Ascoli theorem. 

We claim that
\begin{equation} \label{eq: continuity h^c}
\lim_{i \to \infty}  A(h_{k_i}^c) = A(h^c), \quad \lim_{i \to \infty} \tilde A(h_{k_i}^c) = \tilde A(h^c), \quad \liminf_{i \to \infty} I_C(h_{k_i}^c) \ge I_C(h^c).
\end{equation}
We argue as in the proof of~\eqref{eq: continuity}. The first equality follows from Steiner's formula since $h_{k_i}^c \to h^c$ uniformly on $[0,1]$. The second one follows from \eqref{eq: signed area -} with $h$ and $h_k$, substituted respectively, by $h^c$ and $h_{k_i}^c$, once we notice that the first term in \eqref{eq: signed area -}  vanishes in the limit since $(h_{k_i}^c)' \to (h^c)'$ weakly in $L^1$. Lastly, 
by \cite[Theorem~3.6]{Buttazzo+}, the convexity and lower semi-continuity of $I$ imply that the functional $u \mapsto \int_0^1 I(u(t)) dt$, defined for $u \in L^1$, is sequentially lower semi-continuous in the weak topology on $L^1$. This implies the last assertion in \eqref{eq: continuity h^c}. 

Putting everything together, we see that inequalities \eqref{eq: h^c ineq} follow from \eqref{eq: continuity}, \eqref{eq: ineq convexif}, and~\eqref{eq: continuity h^c}.

We now prove that $h^c$ is a stopping convex curve. We claim that  $\im(h^c)$ is contained in the boundary of its convex hull, which is genuinely two-dimensional since $0 < A(h) \le A(h^c)$. In fact, if our claim does not hold true, then $h^c(t_0) \in \intr(\conv(\im(h^c)))$ for some $t_0 \in [0,1]$. By Carath\'eodory's theorem, there exist distinct $t_1, t_2, t_3 \in [0,1]$ such that $h^c(t_0)$ belongs to the interior of the triangle with vertices $h^c(t_1)$, $h^c(t_2)$, $h^c(t_3)$. Then $h_{k_i}^c(t_0)$ belongs to the interior of the triangle with vertices $h_{k_i}^c(t_1)$, $h_{k_i}^c(t_2)$, $h_{k_i}^c(t_3)$ for all integer $i$ large enough, which is a contradiction because $h_{k_i}^c$ is a stopping convex curve. 

Next we show that $h^c$ moves along the boundary monotonically. Since $\conv(\im(h^c))$ is convex and two-dimensional, we can pick a non-zero point $O$ from the interior of this set. Then $O$ is an interior point of $\conv(\im(h^c_{k_i}))$ for all integer $i$ that are large enough. Consider the function on $[0,1]$, denoted by $\theta$, defined as the angle between $-O$ and $h^c(t)-O$ measured in the clockwise direction, so that $\theta(0)=0$. Similarly, define the angles $\theta_i(t)$ between $-O$ and $h^c_{k_i}(t)-O$. Since each $h^c_{k_i}$ is a stopping convex curve, each $\theta_i$ is non-decreasing on $[0,1]$. Then so is $\theta$, hence $h^c$ is a stopping convex curve, as claimed.

Step 2. Let us prove extended Green's formula~\eqref{eq: Green}.

If $h$ is a stopping convex curve in $AC_0[0,1]$, then so is every $h_k$. Extend it by putting $\bar h_k(t):=h_k(t)$ for $t \in [0,1]$ and $\bar h_k(t):=h_k(1)(2-t)$ for $t \in (1,2]$.  Since $\bar h_k$ may be constant on some intervals, in order to apply the usual Green's formula we shall transform $\bar h_k$ into a simple curve by a time-change that removes these intervals. Define $L_k(t):=\int_0^t \I_{\{\bar  h_k'(s) \neq 0\}}ds$ for $t \in [0,2]$; then $L_k(1)>0$ for all $k$ large enough if we exclude the trivial case where $h$ is constant. Consider the left-continuous inverse of $L_k$, defined by $L_k^{-1}(t):=\inf\{s\ge 0: L_k(s) \ge t\}$ for $t \in (0, L(2)]$ and $L_k^{-1}(0):=0$. This is a piece-wise linear strictly increasing function with jumps. Put $f_k:= \bar h_k \circ L_k^{-1}$, then $f_k([0,L(2)])=\bar  h_k([0,2])$ and $f_k$ is a piece-wise linear closed simple convex curve. Therefore, we can apply Green's formula, which gives
\[
A(h_k) = A(\bar h_k)= A( f_k) = \int_0^{L(2)} (f_k^\bot \cdot  f_k' ) dt = \int_0^2 (\bar h_k^\bot \cdot \bar h_k') dt = \int_0^1 (h_k^\bot \cdot h_k') dt =\tilde A(h_k).
\]
Finally, we obtain $\tilde A(h)= A(h)$ by \eqref{eq: continuity}, as required.

Step 3. Transformation of $h^c$ into a convex curve. 

a) The curve $h^c$ may not be convex only because it can be constant on some subintervals of $[0,1]$. We will remove them by an appropriate time-change, as in Step 2. Consider the set
\[
D:=\big \{t \in (0,1): h^c(t) \equiv  const  \text { on } (t-\varepsilon, t+\varepsilon) \text{ for some } \varepsilon>0\big \},
\]
which is either empty or consists of at most countable number of disjoint non-adjacent open intervals. Put $D^c:=[0,1] \setminus D$, and define the function $L(t):=\int_0^t \I_{D^c}(s) ds$ on $[0,1]$, which satisfies $L(1)>0$. Its left-continuous inverse $L^{-1}$, given by $L^{-1}(t):=\inf\{s\ge 0: L(s) \ge t\}$ for $t >0$ and $L^{-1}(0):=0$, is a strictly increasing function from $[0, L(1)]$ to $[0,1]$. It is measurable as an increasing function. 

We claim that 
\begin{equation} \label{eq: pushforward}
P \circ (L^{-1})^{-1} = \I_{D^c}(t) P(dt);
\end{equation}
in words, the pushforward of the Lebesgue measure $P$ on $[0,1]$ under the mapping $L^{-1}$ is the measure on $[0,1]$ with density $ \I_{D^c}$ with respect to $P$. Indeed, for every $s \in [0,1]$,
\[
P \circ (L^{-1})^{-1}([0,s]) = P(\{t \in [0,L(1)]: L^{-1}(t) \in [0,s]\}) = P([0, L(s)])=L(s),
\]
where in the second equality we used monotonicity of $L$ and the fact that $L(L^{-1}(t))=t$ by continuity of $L$. Thus, $L$ is the distribution function of the measure $P \circ (L^{-1})^{-1}$, which proves~\eqref{eq: pushforward}.

Consider the function $h^c \circ L^{-1}$ on $[0, L(1)]$. It is convex since it is constant on no interval by construction, and it has the same image as $h$. Let us show that it is absolutely continuous. For any $t \in [0, L(1)]$, we have
\begin{align} \label{eq: variable change}
h^c(L^{-1}(t))&= \int_0^{L^{-1}(t)} (h^c)'(s) ds = \int_0^1 \big[ \I_{[0, L^{-1}(t)]}(s)  (h^c)'(s) \big]\I_{D^c}(s) ds  \notag \\
&=\int_0^1  \big[ \I_{[0, L^{-1}(t)]}(L^{-1}(u))  (h^c)'(L^{-1}(u))  \big] du  = \int_0^t (h^c)' \circ L^{-1} du,
\end{align}
where in the second equality we used the change of variable formula combined with \eqref{eq: pushforward}, and in the third equality we used that $\I_{[0, L^{-1}(t)]}(L^{-1}(u)) = \I_{[0,t]}(u)$ by strict monotonicity of $L^{-1}$. By the same argument as in~\eqref{eq: variable change}, we see that $\int_0^1 |(h^c)'| ds =\int_0^{L(1)} |(h^c)' \circ L^{-1}| du < \infty$. Thus, the function $(h^c)' \circ L^{-1}$ is integrable, and from~\eqref{eq: variable change} and Lebesgue's differentiation theorem it follows that $h^c \circ L^{-1}$ is absolutely continuous and $(h^c \circ L^{-1})'=(h^c)' \circ L^{-1}$ a.e.\ on $[0, L(1)]$. This equality implies, by the same argument as in \eqref{eq: variable change}, that
\begin{equation} \label{eq: energy change}
I_C(h^c) = \int_0^{L(1)}I \big(h^c(L^{-1}(u))' \big) du + (1- L(1)) I(0).
\end{equation}

b) We now use the convex function $h^c \circ L^{-1}$ to construct a required convex function $g$ on the whole of $[0,1]$. We will do this differently in two cases. Recall that we made no assumptions on $X_1$. If $X_1$ is integrable, we have $\mu = \E X_1$. Otherwise, let $\mu \in \R^2$ be the unique vector such that for every $u \in \R^2$, it is true that $\E (u \cdot X_1) = u \cdot \mu_1$ if $ u \cdot X_1$ integrable and $u \cdot \mu =0$ otherwise. For example, $\mu =0$ if $ u \cdot X_1$ integrable only when $u=0$. We always have $I(\mu)=0$ by Jensen's inequality, and $I(u) \ge 0$ for all $u \in \R^2$ by $K(0)=0$. 

Case 1: $\mu \neq 0$. Consider the absolutely continuous function  on $[0,1]$ defined by $g_1 := h^c \circ L^{-1}$ on $[0, L(1)]$ and $g_1(t) := h(1) + (t-L(1)) \mu $ on $[L(1), 1]$, where  $h^c \circ L^{-1}(L(1)) = h(1)$. By $\im(h^c) = \im (h^c \circ L^{-1}) \subset \im(g_1)$ and \eqref{eq: energy change}, we have
\begin{equation} \label{eq: h^c < g} 
A(h^c)  \le A(g_1), \qquad I_C(g_1) \le I_C(h^c).
\end{equation}

Since $g_1$ may be non-convex, we need to convexify it. Unlike for Step 1, we can do this explicitly in two steps as follows; see also Figure~\ref{fig: convexification of convexififcation}. First pick a point $t_0 \in [0,L(1)]$ such that $h^c(t_0)$ belongs to one of the two support lines to $\conv(\im(h^c))$ that are parallel to $\mu$ and chosen such that the curve 
\begin{equation*} \label{eq: locally convex h}
g_2(t):=\left\{
                 \begin{array}{ll}
                  h^c(L^{-1}(t)),  & t \in [0, t_0] \\
                  h^c(L^{-1}(t_0))+(t-t_0)\mu, & t\in[t_0, t_0+1-L(1)]\\
				  h^c(L^{-1}(t-1+L(1)))+(1-L(1))\mu, & t\in[t_0+1-L(1), 1]
                 \end{array}
               \right.
\end{equation*}
is convex on $[0, t_0+1-L(1)]$; see Figure~\ref{fig: convexification of convexififcation}.b. For example, if the curve $h^c(t_0)$ is smooth, then either $t_0 \in \{0,L(1)\}$ or $t_0$ is a point where $(h^c)'(t_0)$ is directed along~$\mu$. The image of $g_2$ is obtained by ``inserting'' the segment $\{(1-L(1))s\mu: 0 \le s \le 1\}$ into the image of $h^c\circ L^{-1}$ at the point corresponding to time $t_0$.

\begin{center}
	\parbox{5cm}{
	\begin{center}
		\includegraphics{fig-poly-501.mps}
		
		a) $g_1$
	\end{center}
	}
	\hskip 0.3cm
	\parbox{5cm}{
	\begin{center}
		\includegraphics{fig-poly-502.mps}
		
		b) $g_2$
	\end{center}
	}
	\hskip 0.3cm
	\parbox{5cm}{
	\begin{center}
		\includegraphics{fig-poly-503.mps}
		
		c) $g$
	\end{center}
	}
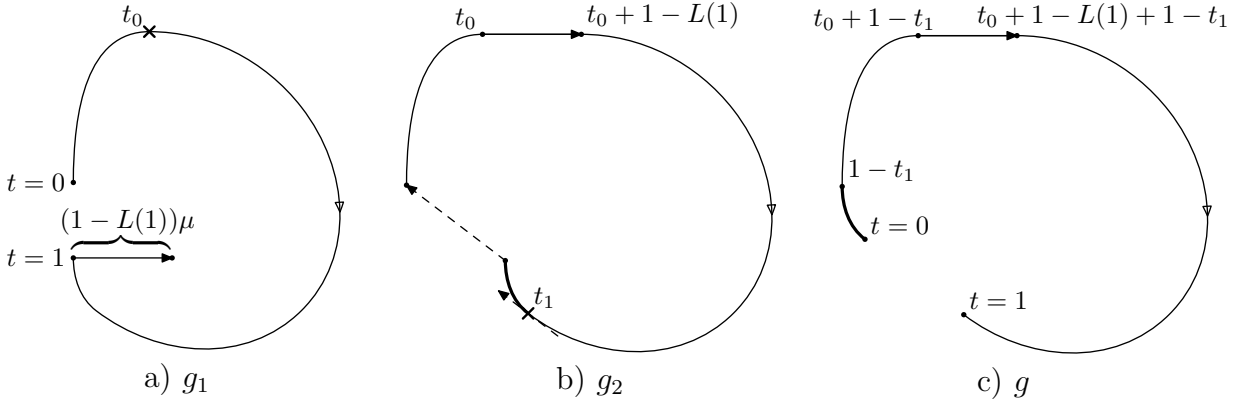
	\f \label{fig: convexification of convexififcation} Convexification of $g_1$. The values shown represent time.
\end{center}


The second step is similar to the previous one, for $- g_2(1)$ instead of $\mu$.
Find a point $t_1 \in [0,1]$ such that the support line to $\conv(\im(g_2))$ at $g_2(t_1)$ is parallel to $g_2(1)$ and $g_2(t_1)$ is directed along  $-g_2(1)$. Move the part of the curve $g_2$ after $t_1$ to the beginning; see Figure~\ref{fig: convexification of convexififcation}.c. This gives a new curve $g$, which is convex and absolutely continuous. We can say that $g$ is a convexification of $g_1$ (and of $g_2$). From  Proposition~\ref{prop:pach theorem} it follows by a discretization argument as in Step 1 that $A(g_1) \le A(g)$ and $I_C(g) = I_C(g_1)$. Then $A(h^c) \le A(g)$ and $I_C(g) \le I_C(h^c)$ by \eqref{eq: h^c < g}. Combined with~\eqref{eq: Green} applied to the stopping convex curves $h^c$ and $g$, the former inequality implies that $|\tilde A(h^c)| \le |\tilde A(g)|$. Combining all these inequalities with the ones in~\eqref{eq: h^c ineq} for $h^c$,  we see that $g$ satisfies inequalities~\eqref{eq: g ineq}, as required. 

Case 2: $\mu =0$. Put $g(t):=h^c \circ L^{-1}(L(1)t)$ for  $t \in [0,1]$. By the change of variables $u = L(1) t$ in~\eqref{eq: energy change} and convexity of $I$,
\[
I_C(h^c) = L(1) \int_0^1 I \big([h^c \circ L^{-1}(L(1)t)]' /L(1) \big) dt \ge  \int_0^1 I \big([h^c \circ L^{-1}(L(1)t)]' \big) dt = I_C(g).
\]
By $\im(g)=\im(h^c)$, we  have $A(g)=A(h^c)$, and also $|\tilde A(g)|=|\tilde A(h^c)|$ by \eqref{eq: Green}. Therefore, $g$ satisfies~\eqref{eq: g ineq}, as required. 

Step 3. It remains to prove the last assertion of the lemma. If $h' \equiv u$ a.e.\ on a non-empty open interval $G \subset [0,1]$ and a non-zero $u \in \R^2$, then for every approximating step function $h_k$, we have  $h_k' \equiv u$ a.e.\ on the interval $G_k$ that is the union of all intervals in $\mathcal D_k$ (i.e.\ the dyadic intervals of length $2^{-k}$)  contained in $G$. Then 
$h_k^c \equiv u$ on a certain translation $H_k$ of the interval $G_k$. Choose an increasing subsequence $k_i$ of $\N$ such that the centres of the intervals $H_{k_i}$ converge as $i \to \infty$. Let $H$ be the open interval of the same length as $G$ centred at the limit of the centres. We have $(h^c)' \equiv u$ a.e.\ on $H$. If $\mu =0$, then $g' \equiv L(1) u$ a.e.\ on $H/L(1)$, otherwise $g'\equiv u$ a.e.\ on a translation of $H$ by construction of $g$ from $h^c$.
\end{proof}

\section{The Euler--Lagrange lemma}
We now state the second auxiliary result needed to prove Theorem~\ref{thm: variational}.

\begin{lem} \label{lem: Lagrangian}
Suppose that $\E e^{u \cdot X_1}< \infty$ for every $u\in \R^2$. 

1. Assume that $\conv( \supp(X_1)) = \R^2$. Then for any $a>0$, 
every minimizer $h$ in \eqref{eq: rate func I_A} is a $C^1$-smooth convex curve that satisfies the Euler--Lagrange equation
\begin{equation} \label{eq: EL main}
\begin{cases}
\lambda h^\bot(t) = \nabla I(h'(t)) - \nabla I(h'(0)) , \quad t \in [0, 1], \\
\nabla I(h'(1)) = -\nabla I(h'(0)),  
\end{cases}
\end{equation}
for some non-zero real $\lambda$.

2. Assume that $\conv( \supp(X_1)) = \{\mu_1\} \times \R$ for a real $\mu_1 \neq 0$. Then for any $a>0$, every minimizer $h=(h_1, h_2)$ in \eqref{eq: rate func I_A} is a $C^1$-smooth convex curve of the form $h(t)=(\mu_1 t, h_2(t))$ such that $h_2$ satisfies the Euler--Lagrange equation
\begin{equation} \label{eq: EL main 1D}
\begin{cases}
\lambda \mu_1 t =I_2'(h_2'(t)) -  I_2' (h_2'(0)), \quad t \in [0, 1], \\
I_2' (h_2'(1)) = -I_2'(h_2'(0)), 
\end{cases}
\end{equation}
for some non-zero real $\lambda$, where $I_2(v):=I(\mu_1, v)$ for $v \in \R$.
\end{lem}

In the proof, we will use the following properties of the rate function $I$. Denote by $\mathcal{D}_{I}$ its effective domain, that is $\mathcal{D}_{I}:=\{u\in \R^2: I(u)<\infty\}$.
\begin{sublem} \label{sublem}
Assume that $\E e^{u \cdot X_1} < \infty$ for every $u\in \R^2$. Then the following is true.
\begin{enumerate}[leftmargin=*, label=\alph*), ref=\alph*]
\item \label{item: I strict} $I$ is {\it strictly} convex on $\mathcal{D}_I$, and 
\begin{equation} \label{eq: D_I inclusions}
\rint(\conv( \supp(X_1)))  \subset \mathcal{D}_{I} \subset  \cl(\conv( \supp(X_1))),
\end{equation}
where $\rint$ denotes the relative interior in the topology of the affine hull of  $\supp(X_1)$.

\item \label{item: nabla I}  If  $\conv( \supp(X_1))=\R^2$, then  $I$ is $C^1$-smooth on $\R^2$, $\nabla I: \R^2 \to \R^2$ is a homeomorphism, and $(\nabla I)^{-1} = \nabla K$.

\item \label{item: I_2'} If $\conv( \supp(X_1))=\{\mu_1\} \times \R$, then  $I_2$ is $C^1$-smooth on $\R$, $I_2' : \R \to \R$ is a homeomorphism, and $(I_2')^{-1} = K_2'$, where $K_2(u):=K(0,u)$ for $u \in \R$.
\end{enumerate}
\end{sublem}
\begin{proof}
For Part~\ref{item: I strict}), see Proposition~1.1 and Corollary~2.12 in Vysotsky~\cite[Proposition~1.1]{VysotskyStrict}. Note that~\eqref{eq: D_I inclusions} holds true under no assumptions on finiteness of the Laplace transform. 

It follows that $K$ is $C^\infty$-smooth on $\R^2$. If  $\conv( \supp(X_1))=\R^2$, then $X_1$ is not supported on a straight line. This yields, by a standard application of H\"older's inequality, that $K$ is strictly convex on $\R^2$. Then the claims of Part~\ref{item: nabla I}) are given by Theorem~26.5 in  Rockafellar~\cite{Rockafellar}. 
The same theorem implies Part~\ref{item: I_2'}) since $I_2$ is the rate function of the $y$-coordinate of $X_1$, whose cumulant generating function is a smooth and strictly convex function $K_2$.
\end{proof}

\begin{proof}[{\bf Proof of Lemma~\ref{lem: Lagrangian}}]
Since the rate function $I$ is finite on $\R^2$ in Part 1 and on $\{\mu_1\} \times \R$ in Part 2, it follows from \eqref{eq: rate func I_A} that $\mathcal J_A$ is finite on $[0, \infty)$. For any $a>0$, we have
\begin{equation}
\label{eq: I_C convex inf}
\mathcal J_A(a)=\min_{\substack{h \in AC_0[0,1]: \\ A(h) \ge a}} I_C(h)= \min_{\substack{ h \in AC_0^{sc}[0,1]: \\ A(h) \ge a}} I_C(h) = \min_{\substack{ h \in AC_0^{sc}[0,1]: \\ |\tilde{A}(h)| \ge a}} I_C(h) = \min_{\substack{h \in AC_0[0,1]: \\  |\tilde{A}(h)| = a}} I_C(h),
\end{equation}
where the first equality follows from \eqref{eq: rate func I_A} and  monotonicity of $\mathcal J_A$ (established in~\cite{AkopyanVysotskyProbab}), the second one follows from Lemma~\ref{lem: convexification} and the inclusion $AC_0^c \subset AC^{sc}_0 \subset AC_0$, in the third one we used~\eqref{eq: Green}, and the last equality is analogous to the first one. We can restate \eqref{eq: I_C convex inf} as follows. Denote by $H_A(a)$ and $\tilde H(a)$, respectively, the sets of minimizers in \eqref{eq: rate func I_A} and in the last minimum in~\eqref{eq: I_C convex inf}. Then by~\eqref{eq: Green}  and Lemma~\ref{lem: convexification}, 
\begin{equation} \label{eq: inclusion for convex}
H_A(a) \cap AC_0^{sc}[0,1] = \tilde H_A(a) \cap AC_0^{sc}[0,1] \neq \varnothing, \quad a>0.
\end{equation}

We split the rest of the proof into two steps.


Step 1. Every $h \in \tilde H_A(a)$ is smooth and solves the Euler--Lagrange equation.

Finding the last minimum in~\eqref{eq: I_C convex inf} is  a standard variational problem of minimizing an integral functional subject to an isoperimetric constraint and a fixed starting point. 

Case 1. $\conv( \supp(X_1)) = \R^2$.

Let $h=(h_1,h_2) \in AC_0[0,1]$ be a minimizer of $I_C$ (not necessarily convex) under either of the constraints $\tilde A(h) = \pm a$. We will use Theorem~\ref{thm: Cesari} from the Appendix with $f(t,p,q)=I(q)$ and $g(t,p,q)=\frac12 p^\bot \cdot q = -\frac12 p \cdot q^\bot$, that is $F=I_C$ and $G=\tilde A$; see~\eqref{eq:signed area}. The function $g$ is $C^1$-smooth, and so is $f$ by Sublemma~\ref{sublem}. Conditions~\eqref{eq: R_0} and~\eqref{eq: R_i} are satisfied for $R(t,q):=\frac12|q|$ since $R(t,h'(t))$ is integrable by $ h \in AC_0[0,1]$. Thus, the assumptions of Theorem~\ref{thm: Cesari} are satisfied, and we conclude that $h$ satisfies the Euler--Lagrange equation 
\[
 \frac{\partial L}{\partial p} (t, h(t), h'(t)) =\frac{d}{dt} \Big ( \frac{\partial L}{\partial q} (t, h(t), h'(t)) \Big ) , \quad \text{a.e. }t \in [0,1],
\]
together with the transversality condition~\eqref{eq: EL boundary cds}  for the time-independent Lagrangian
\[
L(t, p,q) :=  \lambda_1 I(q) + \frac12 \lambda_2 q \cdot {p^\bot}, \qquad p, q \in \R^2,
\]
with some real $\lambda_1$ and $\lambda_2$ such that $\lambda_1^2+\lambda_2^2>0$. 


We have
\begin{equation} \label{eq: dL/dq}
\frac{\partial L}{\partial p}(t, p, q) = -\frac12 \lambda_2 q^\bot, \qquad \frac{\partial L}{\partial q} (t, p, q) = \lambda_1 \nabla I(q) +\frac12 \lambda_2 p.
\end{equation}
The function $\frac{\partial L}{\partial q} (t, h(t), h'(t))$, extended to $[0,1]$ by continuity, is absolutely continuous by Theorem~\ref{thm: Cesari}, and it follows that $\nabla I(h'(t))$ shares the same property because $h$ is absolutely continuous. The Euler--Lagrange equation simplifies to
\[
-\lambda_2 (h')^\bot(t) = \lambda_1 \frac{d}{dt} \big( \nabla I(h'(t)) \big), \quad  \text{a.e. } t \in [0,1].
\]
Note that $\lambda_1 \neq 0$, since otherwise it must be that $h(t) \equiv 0$ by $h(0)=0$, which is impossible by $|\tilde A(h)|=a>0$. Therefore, we can take $\lambda_1=1$. Integrating and using that $h$ and $\nabla I(h'(t)) $ are absolutely continuous and $h(0)=0$, we arrive at
\[
-\lambda_2 h^\bot(t) =\nabla I(h'(t)) - \nabla I(h'(0)), \quad  \text{a.e. } t \in [0,1],
\]
Substituting $-\lambda $ for $\lambda_2$, we obtain the first equality in \eqref{eq: EL main} but so far only for a.e.\ $t$.

Take $t=1$ and use that by the transversality condition \eqref{eq: EL boundary cds}, we have
\[
\nabla I(h'(1)) = -\frac{\lambda_2}{2} h(1)^\bot.
\]
Then $ \nabla I(h'(0))  = -\nabla I(h'(1))$, as stated in \eqref{eq: EL main}. Lastly, it follows that $\lambda_2 \neq 0$, since otherwise $\nabla I(h'(t)) = \nabla I(h'(0))$ for every $t \in [0,1]$, which by Sublemma~\ref{sublem}.\ref{item: nabla I} implies that $h'$ is constant on $[0,1]$, contradicting the condition $|\tilde A(h)| =a >0$. 

To check that \eqref{eq: EL main} holds true for every $t \in [0,1]$, let us show that $h$ is $C^1$-smooth\footnote{Here we prove smoothness using a general method. Later on we will solve the Euler--Lagrange equation explicitly (this does not actually rely on smoothness of $h$), and then check directly that the particular solutions obtained are smooth.}. To this end, we will apply Tonelli {\it regularity theorems}, which ensure smoothness of solutions to the Euler--Lagrange equation. 
We use Theorems 2.6.i and 2.6.ii in~\cite{Cesari}. There are two conditions to check: (i) the gradient mapping $\frac{\partial L}{\partial q} (t, h(t), \cdot)$ is injective for every $t \in [0,1]$, and (ii) $|\frac{\partial L}{\partial q}(t, h(t), q)| \to \infty$ as $|q| \to \infty$ uniformly in $t \in [0,1]$. By \eqref{eq: dL/dq},  we have
\[
\frac{\partial L}{\partial q} (t, h(t), q) = \lambda_1 \nabla I(q) +\frac12 \lambda_2 h^\bot(t).
\]
Condition (i) is satisfied since the rate function $I$ is strictly convex. For Condition~(ii), it suffices to show that $|\nabla I(q)| \to \infty$ as $|q| \to \infty$. Assuming that the sequence $\nabla I(q_n)$ is bounded for some sequence $q_n$ such that $|q_n| \to \infty$, we arrive at a contradiction since $\nabla K ( \nabla I(q_n))=q_n$ by Sublemma~\ref{sublem}.\ref{item: nabla I} but the continuous function $\nabla K$ is bounded on compact subsets of $\R^2$.

Case 2. $\conv( \supp(X_1)) =\{\mu_1\} \times \R$ with $\mu_1 \neq0$, where $\mu=(\mu_1, \mu_2)$. 

We repeat the above argument for the time-dependent Lagrangian
\[
L(t,p,q) :=  I_2(q) + \frac12 \lambda (\mu_1 t q - \mu_1 p ) , \qquad p, q \in \R.
\]
Conditions \eqref{eq: R_i} and~\eqref{eq: R_0} satisfied with $R(t,q)=\frac12 \mu_1 |q|$. From equalities \eqref{eq: EL} and \eqref{eq: EL boundary cds} we obtain that $-\lambda_2 \mu_1 = \lambda_1 I_2''(h_2'(t))$ and $\lambda_1 I_2'(h_2'(1)) = -\frac12 \lambda_2 \mu_1 $, which yield~\eqref{eq: EL main 1D}.

Step 2. Every $h \in H_A(a)$ is convex.

We argue by contradiction. Assume that there exists a $t_0 \in (0,1)$ such that $h(t_0) \in \intr (\conv(\im h)) $. Define the nearest to $t_0$ times of exit from and entrance to the boundary of the convex hull:
\[
t_1:=\max \bigl \{t <t_0: h(t) \in \partial (\conv(\im h)) \bigr \}, \quad 
t_2:=\min \bigl \{t >t_0: h(t) \in \partial  (\conv(\im h) ) \bigr \}.
\]
We have $t_0 \in (t_1, t_2)$ by  continuity of $h$. We claim that $h$ is affine and non-constant  on the interval $[t_1, t_2]$.  Otherwise, by strict convexity of $I$,  the energy $I_C(h)$ of $h$ will decrease if we make $h$ affine on $[t_1, t_2]$. Since this does not change the boundary of the convex hull of $h$ and, consequently, the area $A(h)$, we arrive at a contradiction with  optimality of $h$. 

By Lemma~\ref{lem: convexification}, there is a convex function $g \in H_A(a)$ such that $g'$ is a non-zero constant a.e.\ on an interval. Then $g \in \tilde H_A(a)$ by \eqref{eq: inclusion for convex}, hence $g$ satisfies the Euler--Lagrange equations as shown in Step 1. In Case 1 this implies that $\lambda g$ is constant on the interval, which is a contradiction since $\lambda \neq 0$ and $g$ is convex and hence injective on $[0,1)$. In Case 2 this implies that $\lambda \mu_1 t$ is constant on the interval, which again is a contradiction by $\lambda \neq 0$ and $\mu_1 \neq 0$. 

Thus, the image of $h$ is a subset of $\partial (\conv(\im h))$, and  in order to obtain  convexity of $h$ it remains to prove its injectivity, except for the possible equality $h(0)=h(1)$. Assuming that $h$ is not injective on $(0,1]$, we  find  $0<t_1<t_2\le 1$ such that $\{h(s): t_1\le s\le t_2 \} \subset \{h(s): 0\le s\le t_1 \}$. By the same argument as above (where $t_1$ and $t_2$ were found for a given $t_0$), based on the observation that the convex hull of $\{h(s): 0\le s\le t_2 \}$ will not change if we make $h$ to be affine on $[t_1, t_2]$, we arrive at a contradiction unless $h$ is constant on this interval. Then it follows that $h$ is a stopping convex curve, hence $h \in \tilde H_A(a)$ by \eqref{eq: inclusion for convex}. By Step 1, $h$ satisfies the Euler--Lagrange equation, and this contradicts constancy of $h$ on $[t_1, t_2]$, as already shown. This completes the proof of Step~2.

Putting together the claims of Steps 1 and 2 with equality~\eqref{eq: inclusion for convex} concludes the proof.
\end{proof}

\section{Proofs of the main results}
\begin{proof}[{\bf Proof of Theorem~\ref{thm: variational}}]
Case 1. Assume that $\conv( \supp(X_1)) = \R^2$. 

Recall that $g_{\alpha, \ell}^\tau$ was defined to be the continuous bijective parametrization of the arc $ \{ u \in K^{-1}(\alpha): \tau u \cdot \ell^\bot \ge 0 \}$ that has orientation $\tau$ and satisfies \eqref{eq: parametrization 2}, where $\alpha>0$ . Let us clarify why such parametrization exists, and show that it is $C^\infty$-smooth. Note that the set $K^{-1}(\alpha)$ is a $C^\infty$-smooth one-dimensional manifold since $K$ is $C^\infty$-smooth under the assumption that $\E e^{u \cdot X_1} <\infty$ for every $u \in \R^2$, and $\nabla K \neq 0 $ on $K^{-1}(\alpha)$ because $K$ attains its minimium on $K^{-1}((-\infty, 0])$. Then there exists a $C^\infty$-smooth bijective parametrization $g$ of the above arc  that has orientation $\tau$ and satisfies  $g'(t) \neq 0$ for a.e.\ $t \in [0,1]$.  Define $V(t):= (\lambda_{\alpha, \ell}^\tau)^{-1} \int_{g([0,t])} \frac{\sigma(ds)}{|\nabla K(s)|} $ for $t \in [0,1]$. This function is invertible as a continuous increasing bijection of $[0,1]$ onto itself. Then $g_{\alpha, \ell}^\tau := g \circ V^{-1}$ satisfies \eqref{eq: parametrization 2} by the construction.  It is $C^\infty$-smooth by the inverse function theorem since $V'(t) = (\lambda_{\alpha, \ell}^\tau)^{-1} \frac{|g'(t)|}{|\nabla K(g(t))|}>0 $ by the change of variables formula. 

Let $h \in AC_0[0,1]$ be a minimizer in \eqref{eq: rate func I_A}. By Lemma~\ref{lem: Lagrangian}, it satisfies the Euler--Lagrange equations~\eqref{eq: EL main}. Since $(\nabla I)^{-1} = \nabla K$ by Sublemma~\ref{sublem}.\ref{item: nabla I}, it follows from~\eqref{eq: EL main} that 
\begin{equation} \label{eq:EL 2}
\nabla K(\lambda h^\bot(t) +  \nabla I(h'(0))) =h'(t), \qquad t \in [0,1],
\end{equation}
for some non-zero real $\lambda$. Take the scalar product of both sides of this equality with $(\lambda h'(t))^\bot $, so the left-hand side becomes a total derivative. Then integrate from $0$ to $t$ to arrive at
\begin{equation} \label{eq: constant}
K(\lambda h^\bot(t) + \nabla I(h'(0))) =\alpha , \qquad t \in [0,1],
\end{equation}
where $\alpha := K(\nabla I(h'(0))) $ by $h(0)=0$. 

Denote $g(t):=\lambda h^\bot(t) + \nabla I(h'(0))$ for $t \in [0,1]$. The image of this curve belongs to the level set $K^{-1}(\alpha)$. By Lemma~\ref{lem: Lagrangian}, an optimal trajectory $h$ is a $C^1$-smooth convex curve, hence so is the curve $g$. We have $g(0)= \nabla I(h'(0))$, and  $g(1)=- g(0)$ by the second equality in~\eqref{eq: EL main}. Thus, the set $K^{-1}(\alpha)$  contains two points of opposite sign, which is possible only when $\alpha\ge 0$ because $0 \in K^{-1}((-\infty,\alpha])$ by convexity of $K$. If $\alpha >0$, then $\nabla I(h'(0))$ is non-zero, hence $g(0)$ is the point of intersection of $K^{-1}(\alpha)$ and $\ell \R$ for  $\ell:= \nabla I(h'(0))/|\nabla I(h'(0))|$.

Let us show that $\alpha=0$ is impossible. Otherwise, $g$ must be a closed curve  such that  $ g(1)=g(0)=0$ and $\im g = K^{-1}(0)$. Hence $h'(0)=\nabla K(0)=\mu$. If $\mu=0$, then $K^{-1}(0)=\{0\}$ and hence $h\equiv 0$ by \eqref{eq: constant}, in contradiction to $A(h)=a>0$. If $\mu \neq 0$, then $h$ is cannot be optimal. In fact, consider the curves $h_\varepsilon$ given by $h_\varepsilon(t):=t \mu$ for $t \in [0,\varepsilon]$ and $h_\varepsilon(t):=  \varepsilon \mu + h(t-\varepsilon)$ for $t \in [\varepsilon,1]$, where $\varepsilon \in (0,1)$. We have $I_C(h_\varepsilon) < I_C(h)$ by $I(\mu)=0$.  We also  claim that 
\begin{equation} \label{eq: inclusion}
\conv(\im h) + \varepsilon \mu \subset \conv(\im h_\varepsilon) \text{ for some $\varepsilon>0$ small enough}. 
\end{equation}
Then $A(h) =A(\conv(\im h) + \varepsilon \mu) \le A( h_\varepsilon)$, which contradicts  the optimality of $h$ by monotonicity of $\mathcal J_A$ since $I_C(h_\varepsilon) < I_C(h)$. 

To show~\eqref{eq: inclusion}, we first note that the inclusion shown is equivalent to $\conv (\im g) \subset \conv(g([0,1-\varepsilon]) \cup \{-\varepsilon \lambda \mu^\bot\}) $, which is in turn equivalent to 
\begin{equation} \label{eq: inclusion 2}
\conv(g([1-\varepsilon, 1])) \subset \conv(\{g(1-\varepsilon), 0, -\varepsilon \lambda \mu^\bot\}).
\end{equation}
By convexity of $K^{-1}(0)$, the set on the left-hand side is contained in $\conv(\{g(1-\varepsilon), x_\varepsilon, 0\})$, where $x_\varepsilon$ is the point of intersection of the tangent lines to $K^{-1}(0)$ at points  $g(1-\varepsilon)$ and $0$. Approximating $K^{-1}(0)$ by the osculating circle $C$ at $0$ and using that $g(1-\varepsilon)=-\varepsilon \lambda \mu^\bot + o(\varepsilon)$, one can show that $x_\varepsilon = -\frac12 \varepsilon \lambda \mu^\bot + o(\varepsilon)$ as $\varepsilon \to 0+$. This can be shown to follow from the observation that $x_\varepsilon$ belongs to the straight line passing through the centre of $C$ and the midpoint $g(1-\varepsilon)/2 + o(\varepsilon)$ of the arc of $C$ that approximates the smaller arc of $K^{-1}(0)$ between $0$ and $g(1-\varepsilon)$. For an analytic proof, one can instead consider the Taylor expansion of $K$ to the second order about $0$. Hence $x_\varepsilon \in \conv(\{0,  -\varepsilon \lambda \mu^\bot\})$ for all $\varepsilon >0$ small enough, implying~\eqref{eq: inclusion 2}. Thus, we showed that $\alpha \neq 0$.

Since $\alpha >0$, the curve $g$ bijectively parametrizes either of the arcs $ \{ u \in K^{-1}(\alpha): \pm u \cdot \ell^\bot \ge 0 \}$ and by \eqref{eq:EL 2}, satisfies equality $\nabla K(g(t))= -g'(t)^\bot/\lambda$ for every $t \in [0,1]$. This equality, together with the orientation,  defines uniquely the parametrization $g$ and the constant $\lambda$. This implies that $g=g_{\alpha, \ell}^\tau$ and $\lambda= \tau \lambda_{\alpha, \ell}^\tau$, where $\tau$ is the orientation of $g$, thus establishing~\eqref{eq: optimal curve}. Indeed, we have $\lambda_{\alpha, \ell}^\tau \nabla K(g_{\alpha, \ell}^\tau(t))= -\tau(g_{\alpha, \ell}^\tau)'(t)^\bot$ for $t \in [0,1]$ since (i) the vectors in this equality have the same norm by equivalent definition~\eqref{eq: parametrization} of $g_{\alpha, \ell}^\tau$, and (ii) the vector $\nabla K(g_{\alpha, \ell}^\tau(t))$, normal to the smooth curve $K^{-1}(\alpha)$ at point $g_{\alpha, \ell}^\tau(t)$, is co-directional with $-\tau (g_{\alpha, \ell}^\tau)'(t)^\bot$.  Since $A(h)=a$ and $A(h)=A(g)/\lambda^2$, by definition \eqref{eq: E^pm def} we have $a=  E^\tau(\alpha, \ell)/(\lambda_{\alpha, \ell}^\tau)^2$. By~\eqref{eq: coarea}, this is equivalent to
\[
\frac{1}{2 \sqrt a} = \frac{\lambda_{\alpha, \ell}^\tau}{2\sqrt{E^\tau(\alpha, \ell)}} = \frac{\partial}{\partial \alpha} \sqrt{E^\tau(\alpha, \ell)},
\]
as required. The minimizer $h$ is $C^\infty$-smooth because so is $g_{\alpha, \ell}^\tau$, as we shown above.

Lastly, assume that the distribution of $X_1$ is centrally symmetric. The last expression above then equals $(\sqrt{E(\alpha)/2})'$. Note that for any $\alpha, \beta \ge 0$ and $t \in (0,1)$, we have
\[
E(t \alpha + (1-t) \beta)= A  \big (K^{-1} (t \alpha + (1-t) \beta )\big) > A \big(t K^{-1}(\alpha) + (1-t) K^{-1}(\beta) \big)
\]
by strict convexity of $K$, because if $K(u_1) \le \alpha$ and $K(u_2) \le \beta$ for some $u_1 \neq u_2$, then 
\[
K(t u_1 + (1-t) u_2) < t K(u_1) + (1-t) K(u_2) \le t \alpha + (1-t) \beta. 
\]
On the other hand, by the Brunn--Minkowski inequality,
\begin{align} \label{eq: Brunn-Mink}
\sqrt{A \big(t K^{-1}(\alpha) + (1-t) K^{-1}(\beta) \big)} &\ge \sqrt{ A(t K^{-1}(\alpha))} + \sqrt{ A ((1-t) K^{-1}(\beta) )} \\
& = tE(\alpha)+ (1-t) E(\beta)\notag,
\end{align}
therefore the function $\sqrt{E}$ is strictly concave. Hence the equation $(\sqrt{ E(\alpha)})'=1/\sqrt{2a}$ admits at most one solution for every fixed $a>0$. This solution does exist for every $a>0$ because there always exists an optimal curve $h$ with $A(h)=a$ and $I_C(h)<\infty$.

Case 2. Assume first that $\conv( \supp(X_1)) =\{\mu_1\} \times \R$ with $\mu_1 \neq 0$, where $\mu=(\mu_1, \mu_2)$. Similarly to the argument above in Case 1, it follows from \eqref{eq: EL main 1D} that the $y$-coordinate $h_2(t)$ of an optimal curve $h$ satisfies
\[
K_2'(u(2 t - 1))=h_2'(t), \quad t \in [0,1],
\]
where we denoted $u:=I_2'(h_2'(1))$.  We have $u\neq 0$ since otherwise $h_2 \equiv 0$ and so $A(h)=0$, in contradiction to $a>0$. 
This gives $h_2(t)=\frac{1}{2u} (K(0, u(2t-1)) - K(0,-u) )$. Thus, an optimal curve is of the form given by \eqref{eq: h_pm}, as stated. Its orientation $\tau$ is the sign of $u \mu_1$. Note that replacing $u$ by $-u$ reflects the image of $h$ about the midpoint of the line segment joining $0$ and $h(1)$; recall that $h(t)=(\mu_1 t, h_2(t))$.  Therefore, this changes the orientation of the curve but neither its energy nor the area of its convex hull. Thus, both curves in \eqref{eq: h_pm} are optimal.

It remains to find $u$ using that $A(h)=a$. We have $\tau  \mu_1^{-1}A(h) = \frac12 h_2(1) - \int_0^1 h_2(t)dt$ since if $\mu_1 >0$ and $u>0$ (resp., $u<0$), then the graph of $h_2$, which is the image of $h$ contracted along the $x$-axis by factor $\mu_1$, lies below (resp., above) the line segment joining $0$ and $h(1)$. Then
$$
\frac{a \tau}{\mu_1} = \frac{1}{4u} \Big (K(0, u) + K(0,-u)  - 2\int_0^1 K(0, u(2t-1)) dt \Big) = \frac14 E'(u),
$$
where $E(u)=\int_{-1}^1 K(0,us)ds$. Clearly, $E$ is a symmetric function whose derivative vanishes at $0$ and is continuous and strictly increasing on $[0, \infty)$ by strict convexity of $K(0, \cdot)$. Hence for every $a>0$ and $\tau = \sgn(\mu_1)$, the above equation admits a unique positive solution $u_a$, which exists by existence of an optimal curve for every $a>0$. 
\end{proof}

\begin{proof}[{\bf Proof of Proposition~\ref{prop}}]
Case 1. The distribution of $X_1$ is symmetric and two-dimensional.

Consider the random vector $X_1^{\varepsilon}:= X_1 + \sqrt{\varepsilon} N$, where $\varepsilon >0$  and $N$ is a standard normal random vector on the plane independent of $X_1$. The distribution of  $X_1^{\varepsilon}$ is centrally symmetric and it satisfies $\conv( \supp(X_1^{\varepsilon})) = \R^2$, therefore Theorem~\ref{thm: variational} applies. 
The cumulant generating function $K_\varepsilon$  of this vector satisfies $K_\varepsilon(u)=K(u)+\varepsilon |u|^2/ 2$, and its rate function $I_\varepsilon$ of $X_1^{\varepsilon}$ is given by the infimal convolution 
\[
I_\varepsilon (v)=\inf \big\{ I(v_1) + |v_2|^2/(2 \varepsilon ): v_1 + v_2 =v, v_1 \in \R^2, v_2 \in \R^2 \big\}
\] 
for $v \in \R^2$; see~Rockafellar~\cite[Theorem 16.4]{Rockafellar}. This is the so-called Moreau--Yosida regularization of $I$. It holds that $0 \le I_\varepsilon \le I$ and $I_\varepsilon(v) \to I(v)$ as $\varepsilon \to 0+$ for any $v \in \R^2$. 

Fix an $a \in (0, a_{max})$ and $\ell \in \S$, and let $h_\varepsilon$ be the minimizer in \eqref{eq: optimal curve} for $X_1^{\varepsilon}$ corresponding to the curve $g_{\alpha_\varepsilon, \ell}^+$  with the unique $\alpha_\varepsilon$ given by $(\sqrt{ E_\varepsilon (\alpha_\varepsilon)})'=1/\sqrt{2a}$, where $E_\varepsilon(\beta):=A(K^{-1}_\varepsilon(\beta))$. It is easy to see that for every real $\beta$, the compact convex sets $K_\varepsilon^{-1}((-\infty,\beta])$ converge in the Hausdorff distance to $K^{-1}((-\infty,\beta])$  as $\varepsilon \to 0+$. Hence we have the pointwise convergence $\sqrt {E_\varepsilon(\beta)} \to \sqrt {E(\beta)}$.  Because these functions are strictly concave by the Brunn--Minkowski inequality (cf.~\eqref{eq: Brunn-Mink}), this implies the locally uniform convergence of their derivatives by \cite[Theorem~25.7]{Rockafellar}. Since these derivatives are continuous and strictly decreasing, and by the assumption $1/\sqrt{2a}$ belongs to the interior of the range of $(\sqrt {E(\beta)})'$, it follows that 
\[
\alpha_\varepsilon=[(\sqrt{ E_\varepsilon})']^{-1}(1/\sqrt{2a}) \to [(\sqrt{ E })']^{-1}(1/\sqrt{2a}) =: \alpha
\] 
and $\alpha>0$. This it turn implies convergence of the endpoints $g_{\alpha_\varepsilon, \ell}^+(t) \to g_{\alpha, \ell}^+(t)$ for both $t \in \{0,1\}$ because all these points belong to $\ell \R$. Moreover, by equality~\eqref{eq: coarea},  we have
\begin{equation} \label{eq: lambdas converge}
\lambda_{\alpha_\varepsilon, \ell, \varepsilon}^+ =  \sqrt{E_\varepsilon(\alpha_\varepsilon)} (\sqrt{ E_\varepsilon (\alpha_\varepsilon)})' \to  \sqrt{E(\alpha)} (\sqrt{ E (\alpha)})' = \lambda_{\alpha, \ell}^+ 
\end{equation}
with $\lambda_{\alpha_\varepsilon, \ell, \varepsilon}^+ $ is defined as in~\eqref{eq: lambda^pm} with $K_\varepsilon$ and $\alpha_\varepsilon$ substituted for $K$ and $\alpha$.

Let us prove that $g_{\alpha_\varepsilon, \ell}^+(t) \to g_{\alpha, \ell}^+(t)$ for every $t \in (0,1)$. The convex curves $g_{\alpha_\varepsilon, \ell}^+$ parametrize the arcs $\{ u \in K_\varepsilon^{-1}(\alpha_\varepsilon): \pm u \cdot \ell^\bot \ge 0 \}$, which converge in the Hausdorff distance to the arc $\{ u \in K^{-1}(\alpha): \pm u \cdot \ell^\bot \ge 0 \}$ parametrized by $g_{\alpha, \ell}^+$. On the other hand,  we have $0 \in K_\varepsilon^{-1}((-\infty, \alpha_\varepsilon)) \subset K^{-1}((-\infty, \alpha))$ for all $\varepsilon>0$ that are small enough since $\alpha>0$ and $K_\varepsilon(0)=K(0)=0$. Therefore, $g_{\alpha_\varepsilon, \ell}^+(t) \to g_{\alpha, \ell}^+(t)$ for every $t \in [0,1]$ if and only if $\varphi_\varepsilon(t) \to \varphi(t)$ for every $t \in [0,1]$, where $\varphi(t)$ is the angle between $g_{\alpha, \ell}^+(t)$ and $g_{\alpha, \ell}^+(0)$ and $\varphi_\varepsilon(t)$ is defined analogously. These angle functions are strictly monotone and satisfy $\varphi_\varepsilon(0) = \varphi(0)$ and $\varphi_\varepsilon(1) = \varphi(1)=\pi$, therefore their pointwise convergence is equivalent to pointwise convergence of their inverses. 

For any angle $\theta \in [0, \pi]$, consider the cone 
\[
R^\theta:=\{u \in \R^2: u \cdot \ell \ge |u| \cos \theta, u \cdot \ell^\bot \ge 0 \},
\] 
and denote $E^{\theta}(\beta):=A(K^{-1}(\beta) \cap R^\theta)$ and $E^\theta_{\varepsilon}(\beta):=A(K_\varepsilon^{-1}(\beta) \cap R^\theta)$ for real $\beta$. The functions $\sqrt{E^{\theta}}$ and $\sqrt{E^{\theta}_\varepsilon}$ are strictly concave by the Brunn--Minkowski inequality. Then the same argument as above implies that for any $\theta \in [0, \pi]$, we have
\[
\varphi_\varepsilon^{-1}(\theta) \lambda_{\alpha_\varepsilon, \ell, \varepsilon}^+ = \int_{K_\varepsilon^{-1}(\alpha_\varepsilon) \cap R^\theta} \frac{1}{|\nabla K_\varepsilon(s)|} \sigma(ds) \to \int_{K^{-1}(\alpha) \cap R^\theta} \frac{1}{|\nabla K(s)|} \sigma(ds) = \varphi^{-1}(\theta) \lambda_{\alpha, \ell}^+
\]
as $\varepsilon \to 0+$; this extends \eqref{eq: lambdas converge}, which corresponds to $\theta=\pi$. Hence $\varphi_\varepsilon^{-1}(\theta) \to \varphi^{-1}(\theta)$, as required. As explained above, this yields pointwise convergence of $g_{\alpha_\varepsilon, \ell}^+$ to $g_{\alpha, \ell}^+$. The derivatives of these functions converge pointwise as well because 
\[
(g_{\alpha_\varepsilon, \ell}^+)'(t) = \lambda_{\alpha_\varepsilon, \ell, \varepsilon}^+  \nabla K_\varepsilon(g_{\alpha_\varepsilon, \ell}^+(t)) \to \lambda_{\alpha, \ell}^+  \nabla K(g_{\alpha, \ell}^+(t)) = (g_{\alpha, \ell}^+)'(t)
\] 
using that $\nabla K_\varepsilon$ converge to $\nabla K$ locally uniformly by \cite[Theorem~25.7]{Rockafellar}. Hence for the function $h(t):=-(\lambda_{\alpha, \ell}^+)^{-1} \big(g_{\alpha, \ell}^+(t) - g_{\alpha, \ell}^+(0)\big)^\bot$, it is true that $h_\varepsilon(t) \to h(t)$ and  $h_\varepsilon'(t) \to h'(t)$ for every $t \in [0,1]$.

Finally, note that 
\[
h'(t) = \nabla K(g_{\alpha, \ell}^+(t)) \subset \nabla K(K^{-1}(\alpha)) \subset \nabla K (\R^2) =\intr \mathcal D_I,
\]
where the last equality holds by \cite[Theorem~26.5]{Rockafellar} applied to the differentiable strictly convex function $K$, which is finite on $\R^2$ (and thus is essentially smooth). This implies, using that the pointwise convergence of convex functions $I_\varepsilon \to I$ is uniform on every compact subset of $\intr (\mathcal D_I)$ by \cite[Theorem~10.8]{Rockafellar}, that $I_\varepsilon(h_\varepsilon'(t)) \to I(h'(t))$ for every $t \in [0,1]$. Then for any $g \in AC_0[0,1]$ with $A(g)=a$, we have $I_C(g) \ge {(I_\varepsilon)}_C(g) \ge {(I_\varepsilon)}_C(h_\varepsilon)$ by the inequality $I \ge I_\varepsilon$, hence by Fatou's lemma,
\[
I_C(g) \ge \liminf_{\varepsilon \to 0+} \int_0^1 I_\varepsilon(h_\varepsilon'(t))dt \ge \int_0^1 \liminf_{\varepsilon \to 0+} I_\varepsilon(h_\varepsilon'(t))dt = I_C(h).
\]
Thus, $h$ is a minimizer in \eqref{eq: rate func I_A} that is of the form~\eqref{eq: optimal curve}, as required.  By the same argument, there is a negatively oriented minimizer of the form~\eqref{eq: optimal curve}.

Case 2. $X_1=(\mu_1, Y)$, where $Y$ is a non-constant random variable.

The proof repeats the one in Case 1. Put $K_2(u):=K(0,u)$ for the cumulant generating function of $Y$, and approximate it by $K_2^\varepsilon(u):=K(0,u)+ \varepsilon u^2/2$ for $\varepsilon>0$. The function $E(u)=\int_{-1}^1 K_2(us) ds$ is strictly convex on $\R$ by $E''(u)=\int_{-1}^1 s^2 K_2''(us) ds > 0$ for $s \neq 0$, and so are the functions $E_\varepsilon(u):=\int_{-1}^1 K_2^\varepsilon(us) ds$. We clearly have $E_\varepsilon(u) \to E(u)$ as $\varepsilon \to 0+$ for every $u \in \R$, and this implies, as in Case 1, that  $E_\varepsilon' \to E'$ locally uniformly. Hence the solutions $ u_a^\varepsilon$ to the equations $E'_\varepsilon(u)=4 a /\mu_1$ satisfy $u_a^\varepsilon \to u_a$, where $u_a$ is the solution to $E'(u)=4 a /|\mu_1|$, since $4 a /|\mu_1|$ is in the interior of the range of $E'$ by the assumption. This ensures that  the minimizers $h_\pm^\varepsilon$, defined by substituting $K_2^\varepsilon$ for $K(0, \cdot)$ in~\eqref{eq: h_pm}, converge pointwise to $h_\pm$. This implies, as above, the pointwise convergence $(h_\pm^\varepsilon)' \to (h_\pm^\varepsilon)'$ of derivatives, which in turn implies that $h_\pm$ are minimizers of $I$ subject to $A(h_\pm)=a$. 
\end{proof}

\appendix

\section{Convexifications}

Here we present the planar geometric results needed for our proof of Lemma~\ref{lem: convexification}.

Let $P$ be a directed polygonal line on the plane with vertices $a_1,a_2,\dots, a_n$  passing through these points in the given order\footnote{Formally, $P$ is the equivalence class of continuous curves $f$ such that for some real $ t_1 < \ldots < t_n$, $f$ is defined on $[t_1,t_n]$ and for each $i\in \{1, \ldots, n-1\}$, the restriction of $f$ to $[t_i, t_{i+1}]$ is a bijective parametrization of the line segment $[a_i, a_{i+1}]$ that satisfies $f(t_i)=a_i$. We assume that $a_i \neq a_{i+1}$, and we say that such $f$ parametizes $P$. Let us agree that two planar curves $f$ on $[t_1, t_2]$ and $g: [s_1, s_2] \to \R^2$ are equivalent iff $f = g \circ \gamma$ for some strictly increasing continuous time-change $\gamma:[s_1,s_2] \to [t_1, t_2]$. By $\conv P$ we mean $\conv(\im f)$ for a parametrization $f$ of $P$.}. We denote it by $P=[a_1, \ldots, a_n]$. We will usually omit the adjective `directed', and say that $P$ is closed if $a_1=a_n$. 
A convex polygonal line $[b_1, \dots, b_n]$ is a \emph{convexification} of $P$ if $a_1=b_1$ and there is a permutation $\pi$ of $\{1, \dots, {n-1}\}$ such that 
$a_{i+1} - a_i = b_{\pi(i)+1} - b_{\pi(i)}$ for every $1 \le i \le n-1$. This permutation can be obtained by ordering by increase of the angle between the vectors $a_{i+1} - a_i$ and the direction $a_1-a_n$ if $a_n \neq a_1$ or any fixed direction if $a_n=a_1$. Such a permutation is not unique because of the following reasons. First, taking $\pi'(i):=\pi(n-1-i)$, that is 
measuring the angle in the opposite direction, always gives a different convexification of~$P$ (see Figure~\ref{fig:two polygonal lines}). We call it {\it reverse} to the one given by $\pi$; similarly, we say that $[a_n, \ldots, a_1]$ is the {\it reverse} of $P$. The convexification used in the proof Lemma~\ref{lem: convexification} corresponds to the polygonal line on the left of Figure~\ref{fig:two polygonal lines}. Second, $P$ may have co-directional edges, and any permutation of these edges corresponds to a single edge in a convexification. Third, if the polygonal line is closed, one can replace $\pi$ by its composition with a cyclic permutation, which corresponds to shifting the corresponding vertex of the convexification to $a_1=a_n$. It is easy to see that there are no other convexifications of $P$, and that all of them have the same area of their convex hulls.
\begin{center}
	\parbox{3cm}{
	\begin{center}
		\includegraphics{fig-conv-1.mps}
	\end{center}
	}
	\hskip 2cm
	\parbox{3cm}{
	\begin{center}
		\includegraphics{fig-conv-2.mps}
	\end{center}
	}
	
	\f \label{fig:two polygonal lines} Two convexifications of a polygonal line.
\end{center}

Let us prove the main proposition.
\begin{prop}[J.~Pach]
	\label{prop:pach theorem}
	Let $P_c$ be a convexification of a polygonal line $P$. Then 
	\[
	A(\conv P) \le A(\conv P_c).
	\]
\end{prop}
This proposition was stated in B{\"o}r{\"o}czky et al.~\cite[Proposition 1 and Section (7.5)]{boroczky1986maximal}, where it is also conjectured that the equality holds only if $P$ is either convex or a self-intersecting trapezoid.
The proof is by a slight modification of the one by Pach~\cite{pach1978isoperimetric} for the isoperimetric problem for polygons.
We present it here for completeness of exposition since no proof of exactly this statement was published.
%

\begin{proof}[\bf Proof of Proposition~\ref{prop:pach theorem}]
    Let $P$ be a polygonal line with vertices $a_1,a_2,\dots, a_n$. Without loss of generality we can assume that it is closed.
	Indeed, adding the (directed) edge $a_na_1$ to a convexification $P_c$ of $P$ results is a convexification of the closed polygonal line obtained by adding $a_n a_1$ to $P$. If $P$ contradicts the assertion of the proposition, so does its closed modification.
	
	Furthermore, without loss of generality we may assume that $P$ has a maximal area of its convex hull among all polygonal lines generated by permutations of edges of $P$. This is because permuting the edges does not change convexifications. 
	
	The proof is by induction on the number of vertices $n$. The basis case $n=4$, when $P$ is a triangle with $a_4=a_1$, is trivial. Denote $a_0:=a_{n-1}$.
	
	1. Suppose that some point $a_i$, where $1 \le i \le n-1$, is not a vertex of $\conv P$. 
	Drop this vertex $a_i$ from $P$ and connect  vertices $a_{i-1}$ and $a_{i+1}$ by an edge. Denote this new polygonal line by $P'$, and let $P'_c$ be any of its convexifications.
	Then $A (\conv P'_c)\geq A(\conv P')=A(\conv P)$ by the assumption of induction.
	Replace the side $P'_c$ corresponding to the edge $a_{i-1}a_{i+1}$ of $P'$ by a pair of edges corresponding to $a_{i-1}a_i$ and $a_ia_{i+1}$ taken in such an order that the added triangle (possibly degenerate) does not lie in the interior of $P'_c$. 
	Denote the new polygonal line by $P''$, and note that its area is not less than the area of $\conv P$ (actually, it is strictly larger unless the triangle is degenerate). If $P''$ is convex, then it is a convexification of $P$, which proves the induction step, and hence the proposition, since all convexifications of $P$ have the same area of their convex hulls. If $P''$ is not convex, then $A(\conv P'') > A(\conv P)$, which contradicts the assumptions on~$P$ since $P''$ is obtained by a permutation of edges of~$P$.
	
	2. Now, suppose all vertices of $P$ are the vertices of its convex hull. We use the following result proved in \cite{pach1978isoperimetric}.
	\begin{lem}
		\label{lem:pach-triangle}
		Let $Q$ be a strictly convex polygon.  Suppose that the area of the triangles spanned by the vertices of $Q$ is minimized on some vertices $a,b,c$. Then two sides of $\triangle abc$ lie on the boundary of $Q$.
	\end{lem}
	
Now choose a triangle $\triangle a_ia_ja_k$ of the minimal area among the triangles spanned by the vertices of the convex hull of $P$. Note that by Lemma~\ref{lem:pach-triangle}, $a_i$, $a_j$, $a_k$ are ``consecutive'' vertices on the boundary of $\conv P$, but they are not necessary consecutive for $P$. 	Again, drop the middle vertex $a_j$ from $P$ and connect  vertices $a_{j-1}$ and $a_{j+1}$ by an edge. Denote the new polygonal line by $P'$ and let $P'_c$ be any of its convexifications. 	Replace the side $P'_c$ corresponding to the edge $a_{j-1}a_{j+1}$ of $P'$ by a pair of edges corresponding to $a_{j-1}a_j$ and $a_ja_{j+1}$ in a such an order that the added triangle lies in the exterior of $P'_c$. Denote the new polygonal line by $P''$. Again, applying the induction assumption to $P'$, we get
	\[
		A(\conv P'')\geq A(\conv P'_c)+A(\triangle a_{j-1}a_ja_{j+1})\geq
		A (\conv P')+A(\triangle a_ia_ja_k)=A(\conv P).
	\]
	As in the first part of the proof, if $P''$ is convex, then we get the required inequality.
	If not, then the first inequality above is strict and so $A(\conv P'') > A(\conv P)$, contradicting the assumptions on $P$.
	
\end{proof}	
	
	The following result by Fary and Makai~\cite[Lemma~1]{fary1982isoperimetry} is a version of Proposition~\ref{prop:pach theorem} for the signed areas. These authors define the signed area $\sigma(f)$ enclosed by a closed piecewise linear continuous planar curve $f$ by $\sigma(f) := \int_{\R^2} w(x; f) dx$, where $w(x;f)$ is the winding number of $f$ around $x \in \R^2 \setminus \im h$.  Similarly, we can define the signed area of a (directed) closed polygonal line $P$ by $\sigma(P):=\sigma(f)$ for a parametrization $f$ of $P$; this quantity does not depend on $f$ since $w(x;f )=w(x;g)$ whenever $f$ and $g$ are equivalent. 

	\begin{lem}[I.~Fary and E.~Makai]
		\label{lem:fary-convex}
		Let $P_c$ be a convexification of a closed directed polygonal line~$P$. Then $|\sigma (P)| \le |\sigma (P_c)|$.
	\end{lem}
	
	The actual assertion of~\cite{fary1982isoperimetry} is stated without the absolute values but we can always change the signs of $\sigma(P)$ and $\sigma(P_c)$ replacing $P$ and $P_c$ by their reverses. In fact, the reverse of $P$ has the same convexifications as $P$.	
	
	 In our considerations we used a different definition. Namely, the signed area $\tilde A(h)$ of  a closed curve $h \in AC[t_1,t_2]$ is given by
	 \begin{equation} \label{eq: signed general}
	 \tilde A(h):=\frac12 \int_{t_1}^{t_2} h^\bot(t) \cdot h'(t) dt.
	 \end{equation}
	 This matches our definition~\eqref{eq: signed area -} given for $h \in AC_0[0,1]$ since we can always extend $h$ to a closed curve, say, by putting $h(t):=h(1)(2-t)$ for $t \in (1,2]$, where $h^\bot \cdot h'$ does not contribute to the integral on $(1,2]$.  
	 The following result matches our definition with that of~\cite{fary1982isoperimetry}. 
	 
	 \begin{lem} \label{lem: signed}
	 Let $h$ be a piece-wise linear continuous closed planar curve on $[0,1]$. Then
	 \[
	 \tilde A(h)= \sigma(h).
	 \]
	 \end{lem}
	 We give a proof because we did not find a reference.
	 \begin{proof} 
	 By assumption, $h$ parametrizes some polygonal line  with vertices $a_1, \ldots, a_n, a_1$. Assume that $n \ge 3$, otherwise the claim is trivial by $\tilde A(h)=0$ and $\sigma(h) =0$. Let $g$ be a piecewise linear continuous parametrization of  the polygonal line  with vertices 
   \[
	 a_1, a_2, a_3, a_1, a_3, a_4, a_1, a_4, a_5, a_1, \ldots, a_1, a_{n-1}, a_n, a_1.
	 \]	 
Denote by $t_1< \ldots < t_{n-1}$ the times $t$ when $g(t)=a_1$, and consider the curves $g_i:= g|_{[t_i, t_{i+1}]}$ for $i \in \{1, \ldots, n-2\}$. Each $g_i$ is a closed simple curve, which parametrizes the closed triangular line $[a_1, a_{i+1}, a_{i+2}, a_1]$. Then
	 \[
    \tilde A(g)=\tilde A(g_1) + \ldots + \tilde A(g_{n-2}) = \sigma(g_1) + \ldots + \sigma(g_{n-2})=\sigma(g),
	 \]
	 where in the first equality we used additivity of the integral in~\eqref{eq: signed general}; in the second one we used Green's formula for the curves $g_i$; and in the last equality we used additivity of the signed area, which follows from additivity of the winding numbers:
	 \[
	 w(x; g)= w(x; g_1) + \ldots +  w(x; g_{n-2}), \qquad x \in \R^2 \setminus \im g.
	 \]
	It remains to use the equalities $\tilde A(g)= \tilde A(h)$ and $\sigma(g)=\sigma(h)$, which follow from pairwise cancellations of the integrals (in~\eqref{eq: signed general} and in the definition of a winding number) on the intervals that parametrize $[a_i, a_1]$ and $[a_1,a_i]$ for $i \in \{3, \ldots, n-1\}$.
	 \end{proof}

\section{The Euler--Lagrange equations}

Here we present the result of variational calculus used in the proof of Lemma~\ref{lem: Lagrangian}.

Let $f(t,p,q)$ and $g(t,p,q)$ be $C^1$-smooth real-valued functions, where  $(t,p,q) \in [0,1] \times \R^d \times \R^d$ for $d \in \N$.  We write the gradients of $f$, $g$ in the variables $p$, $q$ as $\frac{\partial f}{\partial p}$, $\frac{\partial f}{\partial q}$, etc. Denote by $AC[0,1]$ is the set of functions $x:[0,1] \to \R^d$ with absolutely continuous coordinates $x_1, \ldots, x_d$, and by $AC_0[0,1]$ its subset satisfying $x(0)=0$. Consider the integral functionals $F(x):=\int_0^1 f(t, x(t), x'(t)) dt$ and $G(u) := \int_0^1 g(t, x(t), x'(t)) dt $ defined for $x \in AC_0[0,1]$ such that $f(t, x(t), x'(t)) $ and $g(t, x(t), x'(t)) $ are integrable, respectively. 

\begin{theorem} \label{thm: Cesari}
Let $x_*$ be a minimizer of  $F(x)$  over $x \in AC_0[0,1]$ such that $G(x)= a$ for some fixed real $a$. Assume that there exist a $\delta >0$  and a continuous function $R: [0,1] \times \R^d \to \R$ such that $R(t, x_*'(t))$ is integrable and  the following conditions \eqref{eq: R_i} and \eqref{eq: R_0} are satisfied:
For every $x: [0,1] \to \R^d$ and $s: [0,1] \to \R$ such that $|x(t)-x_*(t)| \le \delta $ and  $|s(t)-t| \le \delta $  for all $t \in [0,1]$, we have
\begin{equation} \label{eq: R_i} \tag{$R_i$}
\Big|\frac{\partial f}{\partial p} (s(t), x(t), q) \Big| + \Big|\frac{\partial g}{\partial p} (s(t), x(t), q) \Big|\le R(t, q), \qquad t \in [0,1], q \in \R^d,
\end{equation}
and 
\begin{equation} \label{eq: R_0} \tag{$R_0$}
\Big|\frac{\partial  f}{\partial t} (s(t), x(t), q) \Big| + \Big|\frac{\partial  g}{\partial t} (s(t), x(t), q) \Big|\le R(t, q), \qquad t \in [0,1], q \in \R^d.
\end{equation}

Then there exist real $\lambda_1$ and $\lambda_2$ such that $\lambda_1^2+\lambda_2^2>0$ and a function $\xi \in AC[0,1]$ that satisfies the following three equalities:
\begin{equation} \label{eq: xi}
\xi(t) =\frac{\partial L}{\partial q} (t, x_*(t), x_*'(t)), \qquad \text{a.e. }t \in [0,1],
\end{equation}
 for $L:= \lambda_1 f + \lambda_2 g$; the Euler--Lagrange equations
\begin{equation} \label{eq: EL}
\xi'(t) = \frac{\partial L}{\partial p} (t, x_*(t), x_*'(t)), \qquad \text{a.e. }t \in [0,1];
\end{equation} 
and the transversality condition
\begin{equation} \label{eq: EL boundary cds}
\xi(1) =0.
\end{equation}
\end{theorem}

The cumbersome formulation of the Euler--Lagrange equations in terms of $\xi$, which extends $\frac{\partial L}{\partial q} (t, x_*(t), x_*'(t))$ by continuity, is needed to justify differentiation of this function, which a priori is defined only for a.e.\ $t$, and state its value at $t=1$.

Our reference is the book by Cesari~\cite{Cesari}. The result above is assembled from several sections of this book, and we comment on it in detail, especially because the matter appears to be complicated for a non-specialist. In general, an absolutely continuous minimizer may have unbounded derivative, and it may not satisfy the Euler--Lagrange equations. All these difficulties can be ignored if one a priori knows that a minimizer is $C^1$-smooth. However, there may not be such minimizers, or a minimizer over $C^1$ may not be a minimizer over $AC$. Even more, there may be no minimizers in $AC$. 

The base statement for Theorem~\ref{thm: variational} is Theorem 2.2.i in~\cite{Cesari}. Part (a) of this theorem asserts that a minimizer $x_0$ of the integral functional $F(x)$ over functions $x$ with fixed $x(1)$ satisfies the Euler--Lagrange equations under the additional assumption that $x_0'$ is essentially bounded. This assumption is substituted by conditions \eqref{eq: R_i} and \eqref{eq: R_0}  later on in Section~2.7; see also~Eq.~(4.2.10) on p.~168. The assumption that the endpoint $x(1)$ is fixed is removed in Part~(i) of Theorem~2.2.i (cf.~Section~2.8), resulting in the transversality condition. Lastly, Remark~2 on~p.~34 (cf.~Section~4.8) of~\cite{Cesari} reduces our problem of minimization of $F(x)$ subject to the integral constraint $G(x) = a$ to an unconstrained problem of minimizing $\lambda_1 F + \lambda_2 G$, where $\lambda_1$ and $ \lambda_2$ are the Lagrange multipliers. 

Finally, we note that there exist sufficient conditions for smoothness of the minimizers that satisfy the Euler--Lagrange equations. The corresponding results, called {\it regularity theorems}, can be found in Section 2.6 of~\cite{Cesari}. 

\section*{Acknowledgements} I thank Arseniy Akopyan for his tremendous help with the geometric part of the paper. I am also grateful to the three anonymous referees for their valuable comments, which led to further improvement of the paper. This work was supported in part by Dr Perry James (Jim) Browne Research Centre.

\bibliographystyle{plain}
\bibliography{convdev}

\begin{thebibliography}{10}

\bibitem{AkopyanVysotskyProbab}
Arseniy Akopyan and Vladislav Vysotsky.
\newblock Large deviations of convex hulls of planar random walks and brownian
  motions.
\newblock {\em Ann. H. Lebesgue}, 4:1163--1201, 2021.

\bibitem{ATF}
V.~M. Alekseev, V.~M. Tikhomirov, and S.~V. Fomin.
\newblock {\em Optimal control}.
\newblock Consultants Bureau, New York, 1987.

\bibitem{boroczky1986maximal}
K.~B{\"o}r{\"o}czky, I.~B{\'a}r{\'a}ny, E.~Makai, Jr., and J.~Pach.
\newblock Maximal volume enclosed by plates and proof of the chessboard
  conjecture.
\newblock {\em Discrete Math.}, 60:101--120, 1986.

\bibitem{Burchard}
Almut Burchard.
\newblock A short course on rearrangement inequalities.
\newblock 2009.
\newblock Available online at https://www.math.utoronto.ca/almut/rearrange.pdf.

\bibitem{Busemann}
Herbert Busemann.
\newblock The isoperimetric problem in the {M}inkowski plane.
\newblock {\em Amer. J. Math.}, 69:863--871, 1947.

\bibitem{Buttazzo+}
Giuseppe Buttazzo, Mariano Giaquinta, and Stefan Hildebrandt.
\newblock {\em One-dimensional variational problems}.
\newblock The Clarendon Press, Oxford University Press, New York, 1998.
\newblock An introduction.

\bibitem{Cesari}
Lamberto Cesari.
\newblock {\em Optimization -- theory and applications}.
\newblock Springer-Verlag, New York, 1983.
\newblock Problems with ordinary differential equations.

\bibitem{ClarkeVinter}
F.~H. Clarke and R.~B. Vinter.
\newblock Regularity properties of solutions to the basic problem in the
  calculus of variations.
\newblock {\em Trans. Amer. Math. Soc.}, 289(1):73--98, 1985.

\bibitem{Cygan+}
Wojciech Cygan, Nikola Sandri\'c, Stjepan \v{S}ebek, and Andrew Wade.
\newblock Iterated-logarithm laws for convex hulls of random walks with drift.
\newblock {\em Trans. Amer. Math. Soc.}, 377(9):6695--6724, 2024.

\bibitem{EvansGariepy}
Lawrence~C. Evans and Ronald~F. Gariepy.
\newblock {\em Measure theory and fine properties of functions}.
\newblock CRC Press, Boca Raton, FL, revised edition, 2015.

\bibitem{fary1982isoperimetry}
Istv{\'a}n F{\'a}ry and Endre Makai, Jr.
\newblock Isoperimetry in variable metric.
\newblock {\em Studia Sci. Math. Hungar.}, 17(1-4):143--158, 1982.

\bibitem{Figalli+}
A.~Figalli, F.~Maggi, and A.~Pratelli.
\newblock A mass transportation approach to quantitative isoperimetric
  inequalities.
\newblock {\em Invent. Math.}, 182:167--211, 2010.

\bibitem{Fonseca}
Irene Fonseca.
\newblock The {W}ulff theorem revisited.
\newblock {\em Proc. Roy. Soc. London Ser. A}, 432(1884):125--145, 1991.

\bibitem{KVZChambers}
Z.~Kabluchko, V.~Vysotsky, and D.~Zaporozhets.
\newblock Convex hulls of random walks, hyperplane arrangements, and {W}eyl
  chambers.
\newblock {\em Geom. Funct. Anal.}, 27:880--918, 2017.

\bibitem{Khoshnevisan}
Davar Khoshnevisan.
\newblock Local asymptotic laws for the {B}rownian convex hull.
\newblock {\em Probab. Theory Related Fields}, 93:377--392, 1992.

\bibitem{KuelbsLedoux}
James Kuelbs and Michel Ledoux.
\newblock On convex limit sets and {B}rownian motion.
\newblock {\em J. Theoret. Probab.}, 11:461--492, 1998.

\bibitem{MartiniMustafaev}
Horst Martini and Zokhrab Mustafaev.
\newblock On isoperimetric inequalities in {M}inkowski spaces.
\newblock {\em J. Inequal. Appl.}, 2010.
\newblock Art. ID 697954.

\bibitem{Moran}
P.~A.~P. Moran.
\newblock On a problem of {S}. {U}lam.
\newblock {\em J. London Math. Soc.}, 21:175--179, 1946.

\bibitem{Osserman}
Robert Osserman.
\newblock The isoperimetric inequality.
\newblock {\em Bull. Amer. Math. Soc.}, 84:1182--1238, 1978.

\bibitem{Osserman79}
Robert Osserman.
\newblock Bonnesen-style isoperimetric inequalities.
\newblock {\em Amer. Math. Monthly}, 86:1--29, 1979.

\bibitem{pach1978isoperimetric}
J{\'a}nos Pach.
\newblock On an isoperimetric problem.
\newblock {\em Studia Sci. Math. Hungar.}, 13:43--45, 1978.

\bibitem{Rockafellar}
R.~Tyrrell Rockafellar.
\newblock {\em Convex analysis}.
\newblock Princeton University Press, Princeton, N.J., 1970.

\bibitem{DeviationsBook}
Adam Shwartz and Alan Weiss.
\newblock {\em Large deviations for performance analysis}.
\newblock Chapman \& Hall, London, 1995.
\newblock Queues, communications, and computing, With an appendix by Robert J.
  Vanderbei.

\bibitem{Tilli}
Paolo Tilli.
\newblock Isoperimetric inequalities for convex hulls and related questions.
\newblock {\em Trans. Amer. Math. Soc.}, 362(9):4497--4509, 2010.

\bibitem{VysotskyStrict}
Vladislav Vysotsky.
\newblock When is the rate function of a random vector strictly convex?
\newblock {\em Electron.\ Commun.\ Probab.}, 26:article no.\ 41, 2021.

\bibitem{WadeXu2}
Andrew~R. Wade and Chang Xu.
\newblock Convex hulls of random walks and their scaling limits.
\newblock {\em Stochastic Process. Appl.}, 125(11):4300--4320, 2015.

\bibitem{Williams}
David Williams.
\newblock {\em Probability with martingales}.
\newblock Cambridge University Press, Cambridge, 1991.

\bibitem{Zalgaller}
V.~A. Zalgaller.
\newblock Extremal problems on the convex hull of a space curve.
\newblock {\em Algebra i Analiz}, 8:1--13, 1996.

\end{thebibliography}
\end{document}